 \documentclass[a4paper,12pt,oneside,reqno]{amsart}

\usepackage[T1]{fontenc}
\usepackage{hyperref}
\usepackage[headinclude,DIV13]{typearea}
\usepackage{amssymb,amsmath,amsthm,bbm}
\usepackage[usenames,dvipsnames]{xcolor}
\definecolor{violet}{rgb}{0.6,0.4,0.8}
\usepackage{enumerate,multicol,longtable,array}
\usepackage{kantlipsum}
\usepackage{graphicx} 
\usepackage{caption}
\usepackage{mathrsfs}
\usepackage{xcolor}
\usepackage{enumerate}
\usepackage{dsfont}
\usepackage{nicefrac}	
\theoremstyle{plain}
\usepackage{txfonts}
\usepackage{todonotes}

\newtheorem{theorem}{Theorem}[section]
\newtheorem{lemma}[theorem]{Lemma}

\theoremstyle{definition}
\newtheorem{definition}[theorem]{Definition}
\newtheorem{assumption}{Assumption}

{\itshape}{\rmfamily}

\theoremstyle{remark}
\newtheorem{remark}{Remark}

\numberwithin{equation}{section}

\def\d{\mathrm{d}}

\def\<{\langle}
\def\>{\rangle}
\def\a{\alpha}

\def\e{\e}

\def\l{\lambda}

\def\th{\theta}

\def\o{\omega}

\def\R{{\mathbb R}}  
\def\N{{\mathbb N}}  
\def\P{{\mathbb P}}

 \def \e {{\epsilon}}

 \def \th {{\theta}}

 \def \l {{\lambda}}
 \def \d {{\delta}}
 \def \a {{\alpha}}
 \def \o {{\omega}}
 
 \def \th {{\theta}}

 \def\eee{\mathrm{e}}

 \newcommand{\be}{\begin{equation}}
 \newcommand{\ee}{\end{equation}}

\newcommand{\bea}{\begin{eqnarray}}
 \newcommand{\eea}{\end{eqnarray}}
 
\def\TH(#1){\label{#1}}\def\thv(#1){\ref{#1}}
\def\Eq(#1){\label{#1}}\def\eqv(#1){(\ref{#1})}

 \def \1{\mathbbm{1}}

\def \lb {\left(}
\def \rb {\right)}

\usepackage{hyperref}
\begin{document}
%%%%%%%%%%%%%%%%%%%%%%%%%%%%%%%%%%%%%%%%%%%%%%%%%%%%%%
%%%%%%%%%%%%%%%%%%%%%%%%%%%%%%%%%%%%%%%%%%%%%%%%%%%%%%
%%%%%%%%%%%%%%%%%%%%%%%%%%%%%%%%%%%%%%%%%%%%%%%%%%%%%%
%%%%%%%%%%%%%%%%%%%%%%%%%%%%%%%%%%%%%%%%%%%%%%%%%%%%%%
%%%%%%%%%%%%%%%%%%%%%%%%%%%%%%%%%%%%%%%%%%%%%%%%%%%%%%
 \title[PES with phenotypic plasticity]{The Polymorphic Evolution Sequence for Populations with Phenotypic Plasticity}
  
\author[M. Baar]{Martina Baar}
\address{M. Baar\\ Institut f\"ur Angewandte Mathematik\\
Rheinische Friedrich-Wilhelms-Universit\"at\\ Endenicher Allee 60\\ 53115 Bonn, Germany}
\email{mbaar@uni-bonn.de}
\author[A. Bovier]{Anton Bovier}
\address{A. Bovier\\Institut f\"ur Angewandte Mathematik\\
Rheinische Friedrich-Wilhelms-Universit\"at\\ Endenicher Allee 60\\ 53115 Bonn, Germany}
\email{bovier@uni-bonn.de}


\thanks{M. B. is supported by the German Research Foundation through 
the Priority Programme  1590 ``Probabilistic Structures in Evolution''. A.B. is partially supported by the German Research Foundation in 
the Collaborative Research Center 1060 "The Mathematics of Emergent Effects", 
the Priority Programme  1590 ``Probabilistic Structures in Evolution'',
the Hausdorff Center for Mathematics (HCM), and  the Cluster of Excellence ``ImmunoSensation'' at Bonn University.
}
\newcommand{\frakn}{\mathfrak{n}}
%%%%%%%%%%%%%%%%%%%%%%%%%%%%%%%%%%%%%%%%%%%%%%%%%%%%%%
%%%%%%%%%%%%%%%%%%%%%%%%%%%%%%%%%%%%%%%%%%%%%%%%%%%%%%
\begin{abstract}
In this paper we study a  class of stochastic individual-based models 
that describe the evolution of haploid populations 
 where each individual is characterised by a 
phenotype and a genotype. 
 The phenotype of an individual determines its natural birth- and death rates as well as the competition 
 kernel, $c(x,y)$ which describes the induced death rate that an individual of type $x$ experiences due
 to  the presence of an individual or type $y$. 
 When a new individual is born, with a small probability a 
 mutation occurs, i.e. the offspring has  different genotype as the parent. The novel aspect of the 
 models we study is that an individual with a given genotype may express a certain set of different 
 phenotypes, and   during its lifetime it may switch between different phenotypes, with rates that 
 are much larger then the mutation rates and that, moreover, may depend on the state of the entire 
 population. 
The evolution of the population is described by a continuous-time, 
measure-valued Markov
process. 
In \cite{B_MITC},  such a  model was proposed to describe tumor 
evolution under  immunotherapy. 
In the present paper we consider a large class of models which comprises the example studied in 
\cite{B_MITC} and analyse their scaling limits as the population size tends to infinity and the 
mutation rate tends to zero. 
Under suitable assumptions, we prove convergence to a Markov jump process that 
 is  a generalisation of the polymorphic evolution sequence (PES) as 
analysed in \cite{C_TSS,C_PES}.
\end{abstract}
\maketitle
%%%%%%%%%%%%%%%%%%%%%%%%%%%%%%%%%%%%%%%%%%%%%%%%%%%%%%
%%%%%%%%%%%%%%%%%%%%%%%%%%%%%%%%%%%%%%%%%%%%%%%%%%%%%%
\section{Introduction}\label{IntroCancer}
%%%%%%%%%%%%%%%%%%%%%%%%%%%%%%%%%%%%%%%%%%%%%%%%%%%%%%
%%%%%%%%%%%%%%%%%%%%%%%%%%%%%%%%%%%%%%%%%%%%%%%%%%%%%%
%\textcolor{blue}{Compare with the arXiv-version of the Scientific Report paper.}

Over the last decade there has been increasing interest in the mathematical analysis of so-called \emph{stochastic individual based models of adaptive dynamics}. These models were introduced in a series 
of papers by  Bolker, Pacala, Dieckmann, and Law \cite {BolPac1,BolPac2,DieLaw}. They describe the evolution of a population
of individuals characterised by their phenotypes under the influence of the
evolutionary mechanisms of birth, death, mutation, and ecological 
competition in an inhomogeneous "fitness landscape" as a measure valued Markov process. 
In these models there appear two natural scaling parameters. 
The \emph{carrying capacity}, $K$, which regulates the size of the 
population and that can reasonably considered as a large parameter, and 
the mutation rate (of advantageous mutations), $u$, that in many biological
situations can be taken as a small parameter.  In a series of remarkable
 papers, 
 Champagnat and M\'el\'eard
	\cite {C_TSS,C_PES} (and others) have analysed the 
	limiting processes that arise in the limit when $K$ is taken to
	infinity while at the same time $u=u_K$ tends to zero. Under conditions that ensure the separation of   the  \emph{ecological}  and  
	 \emph{evolutionary} time scales.  This   means that the mutation rates are so small that 
 the system has time
	to equilibrate  (ecological time scale) between two mutational events. On the time scale 
	where mutations occur (evolutionary time scale),
	the  evolution of the population can then  be described  as a  
	Markov jump process along a sequence of 
	equilibria of, in general, polymorphic populations.
	An important (and in some sense generic) special case occurs when the mutant population  fixates while the resident 
	population dies out in each step.
The corresponding jump process is called  the \emph {Trait Substitution
		Sequence} (TSS) in adaptive dynamics.  Champagnat \cite{C_TSS} derived criteria in the 
		context of individual-based 
	models under which convergence to the TSS can be proven.
	 The general process is called the \emph {Polymorphic  Evolution Sequence}
	(PES) \cite{C_PES}. Here the limit is describes as a jump process 
	between  possibly polymorphic equilibria of systems of 
	Lotka-Volterra equations of increasing dimension.

In the present paper we extend this analysis to models where an 
additional biological phenomenon is present, the so-called 
\emph{phenotypic plasticity}. By this we mean the following. Individuals 
are no longer described by their phenotype, but by both 
their \emph{genotype} and their phenotype. Moreover, an individual of a given phenotype can express several phenotypes and it can change 
its phenotype during the course of its lifetime. 

Our original 
motivation for this comes from
applications to cancer therapy, where it is well-known that phenotypic switches (\emph{``phenotypic 
plasticity"}) is of utmost importance and in fact a major obstacle to successful therapies
(see, e.g. \cite{Holzel:2013ys} and references therein). For a first attempt at modelling specific 
scenarios in the framework of individual based stochastic models, see \cite{B_MITC}. However, 
phenotypic switches without mutations are certainly relevant in many if not most biological systems.

Here we take a broader look at a large class of models. By expanding the techniques of
 \cite{C_PES} we prove  that  the microscopic process converges  on the evolutionary time scale to a generalisation of the Polymorphic Evolution Sequences (PES) 
 (cf.\ Thm.\ \ref{PESwP}). The main difference in the proof is that we have to couple the process with multi-type branching processes instead of normal branching processes, which leads also to a different definition of invasion fitness in this setting. Note that we gave in  \cite{B_MITC} already heuristic arguments why the process should converge to a Markov jump process. The aim of this paper is to give the rigorous statement and its proof. 

The remainder of this paper is organised as follows.
In Section \ref{sec-model} we define the model, give a pathwise description of the Markov process we are studying and state the convergence towards a quadratic system of ODEs in the large population limit.
In Sections \ref{sec-PES} we consider the case of rare mutations and fast switches. More precisely,  we state the convergence  to  the  Polymorphic Evolution Sequences with phenotypic Plasticity (PESP)  in Subsection \ref{sec-Thm} and prove it in Subsection \ref{PES-prove}.

%%%%%%%%%%%%%%%%%%%%%%%%%%%%%%%%%%%%%%%%%%%%%%%%%%%%%%
%%%%%%%%%%%%%%%%%%%%%%%%%%%%%%%%%%%%%%%%%%%%%%%%%%%%%%
\section{The microscopic  model}\label{sec-model}
%%%%%%%%%%%%%%%%%%%%%%%%%%%%%%%%%%%%%%%%%%%%%%%%%%%%%%
%%%%%%%%%%%%%%%%%%%%%%%%%%%%%%%%%%%%%%%%%%%%%%%%%%%%%%
In this section we introduce the stochastic individual-based model we  analyse (cf.\ \cite{B_MITC, F_MA, C_TSS, C_PES, BanayeMeleard}). The evolutionary process changes populations on a macroscopic level, but the basic mechanisms of evolution, heredity, variation (in our context caused by mutation and phenotypical switching), and selection, act on the microscopic level of the individuals. We describe  
the evolving population as a stochastic system of interacting individuals, where each individual is characterised by its   phenotype and its genotype.
 
Let $l\geq 1$ and  $\mathcal X$ a finite set of the form  $\mathcal 
X=\mathcal G \times \mathcal P$, where $\mathcal G$ is the set of  
genotypes and
 $\mathcal P$ is the set of phenotypes. We call $\mathcal X$ 
the \emph{trait space} of the population. As usual, we introduce a parameter 
 $K\in \mathbb N$, called  the \emph{carrying capacity}. This parameter 
 allows to  scale the population size and can be interpreted as the size of 
 available space or the amount of available resources.
Let $\mathcal M(\mathcal X)$ be the set of finite, non-negative measures 
on $\mathcal X$, equipped with the topology of weak convergence, and 
let $\mathcal M^K(\mathcal X)\subset \mathcal M(\mathcal X)$ be the set of finite point 
measures on $\mathcal X$ rescaled by K, i.e.
\be
\mathcal M^K(\mathcal X)\equiv \left\{\frac 1 K \sum_{i=1}^n \delta_{x_i}\::
\:n\in \N_0,\: x_1,\ldots x_n \in\mathcal X \right\},
\ee 
where  $\delta_{x}$ denotes the Dirac mass at $x\in \mathcal X$.  
We model the time evolution of a population as an $\mathcal 
M^K(\mathcal X)$-valued, 
continuous time Markov  process $(\nu^K_t)_{t\geq 0}$. 
To account for the  process  basic mechanisms of evolution and 
the phenotypic plasticity, 
we introduce the following parameters:
\begin{enumerate}[(i)]
\setlength{\itemsep}{6pt}
\item$b(p)\in\mathbb R_+$ is the \textit{rate of birth} of an individual with phenotype $p\in\mathcal P$.
\item$d(p)\in\mathbb R_+$ is the \textit{rate of natural death} of an individual with with phenotype $p\in\mathcal P$.
\item$c(p,\tilde p)K^{-1}\in\mathbb R_+$ is the \textit{competition kernel}
		which models the competitive pressure an individual with  phenotype  $p\in\mathcal P$ feels from an individual with  phenotype  $\tilde p\in\mathcal P$ and is  inversely proportional to the carrying capacity $K$.
\item $ s_{\text{nat.}}^g(p,\tilde p)\in\R_+$ 
	is the \emph{natural switch kernel} which models the natural switching
	 from  phenotype $p$  to $\tilde p$ of individuals with genotype $g$. 
\item $ s_{\text{ind.}}^g(p, \tilde p) (\hat p)K^{-1}\in\mathbb R_+$ is the \emph{induced switch kernel} which models the switching
	 from  phenotype $p$  to $\tilde p$ of individuals with genotype $g$ induced by an individual with phenotype $\hat p$.
	
	 (Compare with the cytokine-induced switch of \cite{B_MITC}, especially the one of TNF-$\a$ (Tumour Necrosis Factor).)	 
\item$u_K m(g)$ with $u_K, m(g)\in [0, 1]$ is the \textit{probability that a mutation occurs at birth} from an individual with genotype $g\in\mathcal G$, where $u_K$ is a scaling parameter.
\item$M((g,p),(\tilde g, \tilde p))$ is the \textit{mutation law},  i.e.\ if a mutant is born from  an individual 
	 with trait $(g,p)$, then the mutant's trait is $(\tilde g,\tilde p)$ with probability  $M((g,p),(\tilde g,\tilde p))$. 
\end{enumerate}  
Note that most of the parameters depend on the phenotype only and that we explicitly allow that individuals with different genotypes can express the same phenotype and conversely that individuals with the same genotype can express different phenotypes.

\begin{assumption} \label{switch}For simplicity we assume that $ s_{\text{ind.}}^g(p, \tilde p) (\hat p)K^{-1}=0$ for all $\hat p\in \mathcal X$ whenever $ s_{\text{nat.}}^g(p,\tilde p)=0$, i.e. depending on the environment the total switching rate can be larger or smaller but \emph{not} zero or non-zero.   
\end{assumption}
At any time $t\geq0$, we consider a finite population which consist of 
$N_t$ individuals and each individual is characterised its trait  $x_i(t)\in
\mathcal X$. 
The state of a  population  at time $t$ is  the measure
\be
\nu^K_t=\frac 1 K \sum_{i=1}^{N_t}\delta_{x_i(t)}.
\ee

The population process $ \nu^K$ is a
$\mathcal M^K(\mathcal X)$-valued Markov process with infinitesimal generator ${\mathscr  L^K}$, defined, for any bounded measurable function $\phi:\mathcal M^K(\mathcal X) \rightarrow \mathbb R$ and for all $\mu^K\in \mathcal M^K(\mathcal X)$ by
\begin{align}
&\left({\mathscr  L^K}\phi\right)(\mu^K )\\\nonumber
&\:=\sum_{(g,p)\in\mathcal G\times \mathcal P}
		\left(\phi\left(\mu^K +\tfrac{\delta_{(g,p)}}K\right)-\phi(\mu^K )\right) 
			(1-u_K m(g))b(p)K\mu^K (g,p)   \\[0.2em]\nonumber
&\quad+ \sum_{(g,p)\in\mathcal G\times \mathcal P} \sum_{(\tilde g, \tilde p)\in\mathcal G\times \mathcal P}
		\left(\phi\left(\mu^K +\tfrac{\delta_{(\tilde g,\tilde  p)}}K\right)-\phi(\mu^K )\right)  
		u_K m(g) M\big((g,p),  (\tilde g,\tilde p)\big) b(p) K\mu^K (g,p)\\\nonumber
&\quad+ \sum_{(g,p)\in\mathcal G\times \mathcal P}
		\left(\phi\left(\mu^K -\tfrac{\delta_{(g,p)}}K\right)-\phi(\mu^K )\right) 
			\biggl(d(p)+\sum_{\tilde p\in\mathcal P} c(p,\tilde p)\mu^K  (\tilde p) 
					\biggr)K
						\mu^K(g,p)\\[0.2em]\nonumber
&\quad+ \sum_{(g,p)\in\mathcal G\times \mathcal P}\:\sum_{\tilde p\in\mathcal P}
	\left(\phi\left(\mu^K +\tfrac{\delta_{(g,\tilde p)}}K-\tfrac{\delta_{(g,p)}}K\right)-\phi(\mu^K )\right)
	\biggl(s_{\text{nat.}}^g(p,\tilde p)+\sum_{\hat p\in \mathcal P} s^g_{\text{ind.}}(p, \tilde p)(\hat p)\mu^K(\hat p)\biggr)\:
	 K\mu^K (g,p).\nonumber
\end{align}

The first and second terms describe the births (without and with mutation), the third term describes the deaths due to age or competition, and the last term describes the phenotypic plasticity. Observe that the first and second terms are linear (in $\mu^K$ ), but the third and fourth terms are non-linear. The only difference to the standard model is the presence of the fourth term that corresponds to the phenotypic switches. However, this term changes the dynamics substantially. In particular, the system of differential equations which arises in the large population limit without mutation ($u_K=0$) is not a generalised Lotka-Volterra system anymore, i.e.\ has \emph{not} the form $\dot{\mathfrak n}= \mathfrak n f(\mathfrak n)$, 
where $f$ is linear in $\mathfrak n$ (cf.\ Thm.\ \ref{det-limit} and Def.\ \ref{LVS}). 

\begin{remark}\label{TraitSpace}
\begin{enumerate}[(i)]
\setlength{\itemsep}{0.5pt}
\item 
Since $\mathcal X$ is finite, we could also represent the population state as 
an $|\mathcal X|$-\-di\-men\-sion\-al vector. More precisely,  let $E$ be a subset of $\mathbb R^{|\mathcal X|}$ and $E^K\equiv E\cup\{n/K: n\in \mathbb N_0\}$, then for fixed $K\geq 1$, the population process can be constructed as Markov process with state space $E^K$ by using independent standard Poisson processes (cf.\ \cite{E_MP} Chap.\ 11). 
\item 
For an extension to a non-finite trait space, e.g.\ if $\mathcal G$ and $\mathcal P$ are compact subsets of $\mathbb R^k$ for some $k\geq 1$, the modeling of switching the phenotype has to be changed in the following way:
Each individual with trait $(g,p)\in \mathcal G\times \mathcal P$  has 
instead of the natural switch kernel $s_{\text{nat.}}^g(p,\tilde p)$ a natural switch rate $s_{\text {nat.}}(g,p)$ combined with a probability measure  $ {S}_{\text {nat.}}^{(g,p)}( d \tilde p)$  on $\mathcal P$ and 
instead of the induced switch kernel $s_{\text{ind.}}^g(p,\tilde p)(\hat p)K^{-1}$ a induced switch kernel $s_{\text {ind.}}((g,p),\hat p)K^{-1}$ combined with a family of probability measure  $ \{{S}_{\text {ind.}}^{((g,p), \hat p)}( d \tilde p)\}$  on $\mathcal P$.
\end{enumerate}
\end{remark}

%%%%%%%%%%%%%%%%%%
\subsection{Explicit construction of the population process with phenotypic plasticity}\label{construction}
%%%%%%%%%%%%%%%%%%%%%%%%%%%%%%%%%%%%

 It is useful to give  a pathwise description of $ \nu^K$ in terms of Poisson point measures (cf.\ \cite{F_MA}).
Let us recall this construction.
Let $(\Omega,\mathcal F,\mathbb P)$ be an abstract probability space. 
On this space, we define the following independent random elements:
\begin{enumerate}[(i)]
\setlength{\itemsep}{6pt}
\item a convergent sequence $(\nu^K_0)_{K\geq 1}$ of $\mathcal M^K(\mathcal X)$-valued random measures  (the random initial population), 
\item $|\mathcal X|$ independent Poisson point measures
$(\,N^{\text{birth}}_{(g,p)}(ds,di,d\theta)\,)_{(g,p)\in\mathcal X}$ on 
$\,[0,\infty)\times \mathbb N\times \mathbb R_+\,$ with intensity measure
 $ds\sum_{n\geq 0}\delta_n(di)d\theta$,
\item$|\mathcal X|$ independent Poisson point measures $(\,N^{\text{mut.}}_{(g,p)}(ds,di,d\theta,dx)\,)_{(g,p)\in\mathcal X}$ on 
$\,[0,\infty)\times \mathbb N\times \mathbb R_+\times \mathcal X\,$ 
with intensity measure
$\,ds\sum_{n\geq 0}\delta_n(di)d\theta\sum_{\tilde x\in\mathcal X}\delta_{\tilde x}(dx)$.
\item $|\mathcal X|$ independent  Poisson point measures 
	$(\,N^{\text{death}}_{(g,p)}(ds,di,d\theta)\,)_{(g,p)\in\mathcal X}$ on 
$\,[0,\infty)\times \mathbb N\times \mathbb R_+\,$ with intensity measure
 $\,ds\sum_{n\geq 0}\delta_n(di)d\theta$, 
 \item $|\mathcal X|$ independent Poisson point measures 
	$(\,N^{\text{switch}}_{(g,p)}(ds,di,d\theta, dp)\,)_{(g,p)\in\mathcal X}$ on 
$\,[0,\infty)\times \mathbb N\times \mathbb R_+\times \mathcal P\,$ with intensity measure
 $\,ds\sum_{n\geq 0}\delta_n(di)d\theta\sum_{\tilde p\in\mathcal P}\delta_{\tilde p}(dp)$, 
\end{enumerate}
Then, $\nu^K$ is given by the following equation
\begin{align}
 \nu_t^K
=  \:&\nu_0^K +\!\sum_{(g,p)\in \mathcal X}\int_0^t \int_{\mathbb N_0}\int_{\mathbb R_+}
	\mathds 1_{\left\{i\leq K \nu^K_{s-}(g,p),\; 	
				\theta \leq b(p)\lb1-u_K m(g)\rb\right\}}
	 \tfrac{1}{K}{\d_{(g,p)} } N^{\text{birth}}_{(g,p)}(ds,di,d\th)\\\nonumber
 &+\sum_{(g,p)\in \mathcal X}	\int_0^t \int_{\mathbb N_0}\int_{\mathbb R_+}\int_{\mathcal X}
\mathds 1_{\left\{i\leq K \nu^K_{s-}(g,p),\; \theta \leq b( p)
u_K m( g) M( (g,p),x)\right\}}
 \tfrac{1}{K}{\d_{x}}  N_{(g,p)}^{\text{mut.}}(ds,di,d\th, dx)\\\nonumber
&
 -	\sum_{(g,p)\in \mathcal X}\int_0^t \int_{\mathbb N_0}\int_{\mathbb R_+}
	\mathds 1_{\left\{i\leq K \nu^K_{s-}(g,p),\; 
					\theta \leq d (p)
						+\sum_{\tilde p\in \mathcal P}c( p,\tilde p )  \nu^K_{s^-}(\tilde p)\right\}}
		  \tfrac{1}{K}{\d_{(g,p)}}N_{(g,p)}^{\text{death}}(ds,di,d\th)	\\	\nonumber
 &+	\sum_{(g,p)\in \mathcal X} \int_0^t \int_{\mathbb N_0}\int_{\mathbb R_+}\int_{\mathcal P}
{\mathds 1_{\left\{i\leq K  \nu^K_{s-}(g,p),\; \theta \leq s_{\text{nat.}}^g(p,\tilde p)+\sum_{\hat p\in \mathcal P} s^g_{\text{ind.}}(p, \tilde p)(\hat p)\nu_{s-}^K(\hat p)\right\}}}
\\\nonumber &\hspace{6cm}\times
\tfrac{1}{K}\lb{\d_{(g, \tilde p)}-\d_{(g,p)}} \rb N_{(g,p)}^{\text{switch}}(ds,di,d\th, d\tilde p).
\end{align}
\begin{remark} This construction uses  that $\mathcal X$ is a discrete set and is in some sense closer to the definition given in \cite{E_MP} (p.\ 455).  For non-discrete trait spaces the process can be constructed  as in \cite{F_MA}.
\end{remark}

%%%%%%%%%%%%%%%%%%
\subsection{The Law of Large Numbers.}
%%%%%%%%%%%%%%%%%%%%%%%%%%%%%%%%%%%%
If the  mutation rate is independent of $K$ and the initial conditions converge to a deterministic limit, then the sequence of rescaled processes, $(\nu_{}^K)_{K\geq 1}$, converges, almost surely, as 
$K\uparrow  \infty$ to 
the solution of  system of ODEs. This follows directly from the law of large numbers for density depending processes, see, e.g. Ethier and Kurtz  \cite{E_MP}, Chap.\ 11. The following theorem gives a precise statement.

\begin{theorem}\label{det-limit} Let  $u_K\equiv  {1}$.
	Suppose that the initial conditions converge almost surely to a deterministic
	limit, i.e.\ $\lim_{K\uparrow \infty} \nu^K_0=\nu_0$, where $\nu^{}_0$ is a finite measure on $\mathcal X$. 
	Then, for every $T>0$, exists a deterministic function $\xi\in C([0,T],\mathcal M_F(\mathcal X))$ such that  
	\be
	\lim_{K\uparrow \infty}\sup_{t\in[0,T]}\big|\big| \nu^K_t-\xi_t  \big|\big|^{}_{\text{TV}}=0 ,\qquad \text{a.s.,}
	\ee
	where $||\,.\,||{}_{\text{TV}}$ is the total variation norm.
	Moreover, let $\mathfrak n$ be the
	unique solution to the  dynamical system 
\begin{align}
	\quad\dot{\mathfrak n}_{(g,p)}(t)
	\:=&\;
	\frakn_{(g,p)}(t)
      	  \:\Biggl(\bigl(1-m(g)\bigr)b(p) -d(p)
		 -\sum_{(\tilde g, \tilde p)\in \mathcal G\times\mathcal P}  c(p,\tilde p) \frakn_{ (\tilde g,\tilde p)}(t)\\\nonumber&\hspace{2.5cm}
		-\sum_{\tilde p\in\mathcal P} \biggl( s_{\text{nat.}}^g(p,\tilde p)+\sum_{(\hat g,\hat p)\in\mathcal G\times \mathcal P} s^g_{\text{ind.}}(p, \tilde p)(\hat p) \frakn_{ (\hat g,\hat p)}(t)\biggr)
		\Biggr) \nonumber\\\nonumber
	&+\quad\sum_{\tilde p\in\mathcal P}  \;\;\frakn_{(g,\tilde p)}(t)\:\biggl( s_{\text{nat.}}^g(\tilde p, p)+\sum_{(\hat g,\hat p)\in\mathcal G\times \mathcal P} s^g_{\text{ind.}}(\tilde p,  p)(\hat p) \frakn_{ (\hat g,\hat p)}(t)\biggr)
	\\&\nonumber
	+\sum_{(\tilde g, \tilde p)\in \mathcal G\times\mathcal P}\frakn_{(\tilde g,\tilde p)}(t) \:m(\tilde g)
	 b(\tilde p) M((\tilde g,\tilde p),( g,p)), \quad  (g,p)\in\mathcal G\times \mathcal P,
	\end{align}
 with initial condition $\mathfrak n_x(0)=\nu_0(x)$  for all $x\in \mathcal X$.\\[0.5em]			
Then, $\xi$ is given as $\xi_t=\sum_{x\in \mathcal X}\frakn_x(t)\delta_x$.
\end{theorem}
\begin{proof}
This result follows from 
Theorem 2.1 in Chapter 11 of \cite{E_MP}, since we can construct the process as described in Remark \ref{TraitSpace} (i).
For more details see \cite{DISS_Hannah}.
\end{proof}
\begin{remark}If the trait spaces is not finite, one can obtain a similar result, cf.\ \cite{F_MA}.% Thm 5.3.
\end{remark}

%%%%%%%%%%%%%%%%%%%%%%%%%%%%%%%%%%%%%%%%%%%%%%%%%%%%%%
%%%%%%%%%%%%%%%%%%%%%%%%%%%%%%%%%%%%%%%%%%%%%%%%%%%%%%
\section{The interplay between rare mutations and fast switches.}\label{sec-PES}
%%%%%%%%%%%%%%%%%%%%%%%%%%%%%%%%%%%%%%%%%%%%%%%%%%%%%%
In this section we state our main results. As in previous work, 
we place ourselves under the assumptions 
\begin{equation}\label{conditions}
\forall V>0, \qquad \exp(-VK)\ll u_K \ll \frac{1}{K\ln K}, \qquad \text{as } K\uparrow\infty,
\end{equation}
which ensure that a population reaches equilibrium before a new mutant appears. 
Under these assumptions we prove that the individual-based process with phenotypic plasticity convergences to a generalisation of the PES.
Let us start with describing the techniques  used in \cite{C_PES}.

%\qu
{The key element in the proof of the convergence to the PES used by Champagnat and M\'el\'eard 
\cite{C_PES} is a precise analysis of how a mutant population fixates. 
A crucial assumption in  \cite{C_PES}  is that  the competitive 
Lotka-Volterra systems that describes  the large population limit 
always have a unique stable 
fixed point $ \bar \frakn$. Thus, the main task is to \emph{study the invasion of a mutant} that has just appeared in a
population close to equilibrium. The invasion can be divided into three steps:}
%\qu
{First, as long as the
mutant population size is smaller than $K \epsilon$, for a fixed small $\epsilon>0$, the resident population
stays close to its equilibrium. 
 Therefore, the mutant population can be approximated by a branching 
 process. 
 Second, once the mutant population reaches the level $K \epsilon$, the 
 whole system
is close to the solution of the corresponding deterministic system and 
reaches an $\epsilon$-neighbourhood of $ \bar \frakn$ in finite time. 
Third, the subpopulations which have a zero coordinate in $ \bar \frakn$ 
can be approximated  by subcritical branching processes until they die 
out.}

{The first and third steps require a time of order  $\ln(K)$, whereas the 
 second step requires only a time of order one,  independent of $K$.  
 Since the expected time between two mutations is of order $1/(u_K K)$, 
 the the upper bound on $u_K$  in \eqref{conditions} guarantees that, with 
 high
probability, the three steps of an invasion are completed before a new 
mutation occurs.}

%\qu
{In the first invasion step the invasion fitness of a mutant plays a crucial role. 
  Given a population in a stable equilibrium  that populates a certain set
  of traits, say $M\subset \mathcal X$, the invasion fitness $f(x,M)$ is the growth rate of a population consisting of a single individual 
  with trait $x\not\in M$ in the presence of the equilibrium population $ \bar\frakn$ on $M$. 
  In the case of the standard model, it is given 
  by 
  \begin{equation}
  f(x,M)= b(x)-d(x)-\sum_{y\in M} c(x,y) \bar\frakn_y.
  \end{equation}
Positive $f(x,M)$ implies that a mutant appearing with trait $x$ from the equilibrium 
  population on $M$ has a 
  positive probability (uniformly in $K$) to grow to a population of size of order $K$; 
  negative invasion fitness implies that such a mutant 
  population will die out with probability tending to one (as $K\uparrow \infty$) before it can reach a size of order $K$.
The reason for this is that the branching process (birth-death process) which approximates the mutant population is supercritical if  $f(x,M)$ is positive and subcritical if $f(x,M)$ is negative.}
  
In order to describe the dynamics of a phenotypically heterogeneous 
population on the evolutionary time scale,  we have to adapt the notion of 
invasion fitness to the case where fast phenotypic switches are present.  
Since switches between phenotypes associated to the same genotype 
happen at times of order one, the growth rate of the initial mutant 
phenotype does not determine the probability of fixation. See 
\cite{ColMelMet} for a similar issue in a model with  sexual reproduction. 
In the proof of Theorem \ref{PESwP} we  approximate the dynamics
of the mutant population  by a \emph{multi-type branching process} until 
the reaches a size  $K \epsilon$ (or dies out).
A continuous-time multi-type branching process
 is supercritical if and only if the largest eigenvalue of its infinitesimal 
 generator of  is strictly positive (cf.\ \cite{A_BP, SEW}).
 Therefore, this  eigenvalue will provide an appropriate generalisation of the invasion fitness.
%%%%%%%%%%%%%%%%%%%%%%%%%%%%%%%%%%%%%%%%%%%%%%%%%%%%%%
%%%%%%%%%%%%%%%%%%%%%%%%%%%%%%%%%%%%%%%%%%%%%%%%%%%%%%
\subsection{The competitive Lotka-Volterra system with phenotypic plasticity.}

We first consider 
 the large population limit  without 
 mutation $(u_K\equiv 0)$. 
Assume tha the initial condition is supported on $d$ traits,
 $(\mathbf g,\mathbf p)=((g_1,p_1),\ldots, 
(g_d,p_d))\in(\mathcal G\times\mathcal P)^d$, and  that the sequence of 
the initial conditions converges almost surely to a deterministic limit, i.e.\ 
\be 
\lim_{K\uparrow \infty}  \nu_0^K=\sum_{i=1}^d n_i(0)\delta_{(g_i,p_i)},  \quad 
\text{a.s.,}\qquad \text{ where $ n_i(0)>0$ for all $i\in\{1\ldots d\}$. }
\ee
By Theorem \ref{det-limit}, for every $T>0$, the sequences of processes 
$ \nu^K\in \mathbb D([0,T],\mathcal M^K(\mathcal X))$  generated by
 $ {\mathscr L}^K$ with initial state $ \nu_0^K$
converges almost surely,  as $K\uparrow \infty$, to a deterministic function $\xi\in C([0,T],\mathcal M(\mathcal X))$.
Since $u_K\equiv 0$, no new genotype can appear in the population process $\nu^K$. Moreover, not every genotype can express every phenotype. <let us describe the support of $\nu_t^K$  more precisely.

For all $g\in \mathcal G$,
let $X^g$ be  a  stationary discrete-time Markov chain with state space  
$\mathcal P$ and  transition probabilities 
\be
\mathbb P[X^g_i=\tilde p \:|\: X^g_{i-1}=p]=\frac{s^g(p,\tilde p)}{\sum_{\hat p \in \mathcal P}s^g(p, \hat p)}, \quad\text{ if }\sum_{\hat p \in \mathcal P}s^g(p, \hat p)>0
\ee
and 
\be
\mathbb P[X^g_i= p \:|\: X^g_{i-1}=p]=1, \quad\text{ if }\sum_{\hat p \in \mathcal P}s^g(p, \hat p)=0.
\ee
The Markov chains $\{X^g, g\in\mathcal G\}$ contain only partial 
information on the switching behaviour of the process $\nu^K$, but we 
see that this is  the key information needed later.

In the sequel we work under the following simplifying assumption:
\begin{assumption}\label{recurrence} For all $g\in \mathcal G$, all communicating classes of $X^g$ are recurrent. 
\end{assumption}

 We denote the communicating class 
associated with $(g,p)\in\mathcal G\times \mathcal P$ by $[p]_g$. This is 
the communicating class of $X^g$ which contains $p$, i.e.\ $p$ can be 
seen as  a representative of the class, which has an equivalence relation 
depending on $g$.
\begin{figure}[h!]
	\centering
	 \includegraphics[width=0.5\textwidth]{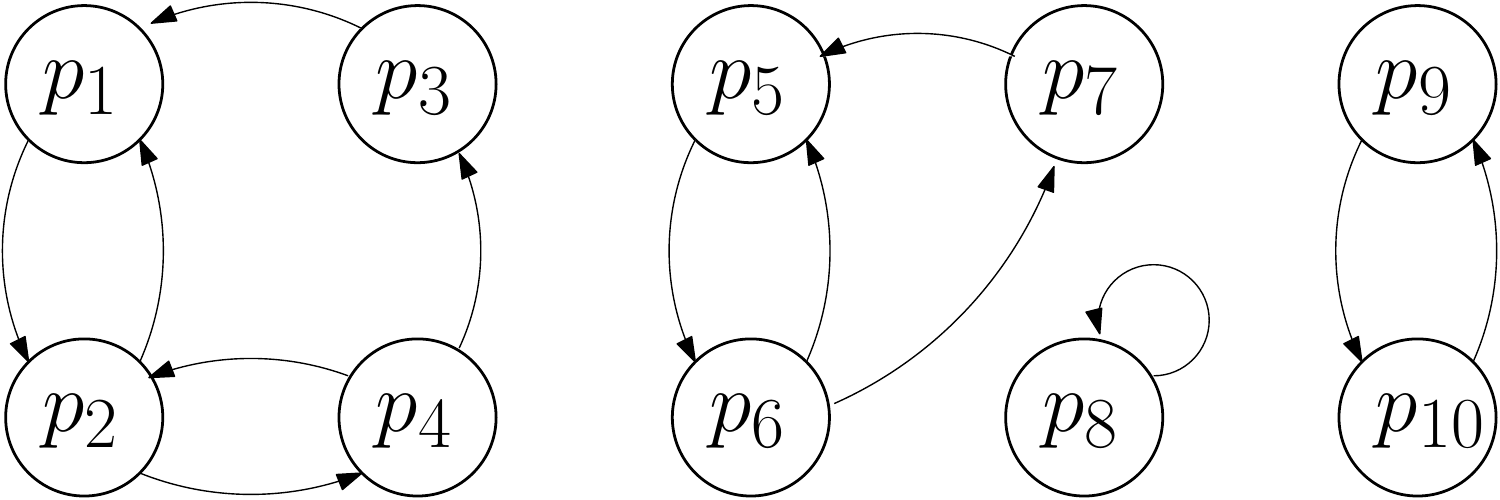}
	 \caption{\small Example of a Markov chain $X^g$. Here, $\mathcal 
	 P=\{p_1,\ldots,p_{10}\}$ and 
	 $X^g$ has four communicating classes:
	 $\{p_1,p_2,p_3,p_4\},\{p_5,p_6,p_7\},\{p_8\},\{p_9,p_{10}\}$. The 
	 class $\{p_8\}$ has only one element, i.e.\ $\sum_{i=1}^{10}
	 s^g(p_8,p_i)=0$ in this example. }	  
\end{figure}
By Assumption \ref{switch},
this ensures that if we start with a large enough population consisting only 
of individuals carrying the same trait $(g,p)$, then, after a short time, all 
phenotypes in the class $[p]_g$ will be present in the population, but 
none of the other classes. 
Observe that these Markov chains do not describe the dynamics of the 
whole process. If we allowed transient states this would not imply that 
the trait would get extinct, since its growth rate could be larger than the 
switching rate.

Thus, $[p]_{g}$ is the set of phenotypes which are reachable in the 
Markov chain $X^g$ with $X^g_0=p$ and
\emph{the set of traits which can appear in the population process $
\nu^K$} is given by
\be
\mathcal X_{(\mathbf g,\mathbf p)}\equiv \bigcup_{i=1}^d\{g_i\}\times 
[p_i]_{g_i}.
\ee

With this notation, $\xi$ is given by $\xi(t)=\sum_{x\in \mathcal X_{(\mathbf g,\mathbf p)}}  {\mathfrak n}_{x} (t)\delta_{x}$, where $ {\mathfrak n}$ is the solution of the competitive Lotka-Volterra system with phenotypic plasticity defined below.

\begin{definition} \label{LVS}
For any $ (\mathbf g,\mathbf p)\in(\mathcal G\times\mathcal P)^d$, we denote by $LVS(d,  (\mathbf g,\mathbf p))$
the  \emph{competitive Lotka-Volterra system with phenotypic plasticity}. This is an $|\mathcal X_{(\mathbf g,\mathbf p)}|$-dimensional system of ODEs given 
 by  
\begin{align}\label{deterministic-system-initial}
\dot{  {\frakn}}_{(g,p)}
=\:& \frakn_{(g,p)} 
	\biggl( b(p)-d(p)
	-\!\!\sum_{(\tilde g,\tilde p)\in \mathcal X_{(\mathbf g,\mathbf p)}}\!\!c(p,\tilde p)\frakn_{(\tilde g,\tilde p)}
	-\!\! \sum_{\tilde p \in[p]_{g}}
		\biggl( s_{\text{nat.}}^g(p,\tilde p)+\!\!\sum_{(\hat g,\hat p)\in\mathcal X_{(\mathbf g,\mathbf p)}} 
				\!\!s^g_{\text{ind.}}(p, \tilde p)(\hat p) \frakn_{ (\hat g,\hat p)}(t)
		\biggr)
	 \biggr)\nonumber\\
    &+\!\! \sum_{\tilde p \in[p]_{g}}\frakn_{(g,\tilde p)}
    		\biggl( s_{\text{nat.}}^g(\tilde p, p) +\!\!\sum_{(\hat g,\hat p)\in\mathcal X_{(\mathbf g,\mathbf p)}}
				 \!\!s^g_{\text{ind.}}(\tilde p,  p)(\hat p) \frakn_{ (\hat g,\hat p)}(t)
		\biggr), \quad (g, p)\in\mathcal X_{(\mathbf g,\mathbf p)}.
\end{align}
\end{definition} 
We choose the (possibly misleading) name \emph{competitive Lotka-Volterra system with phenotypic 
plasticity} to emphasise that we add phenotypic plasticity (induced by 
switching rates) in the usual competitive Lotka-Volterra system. 
Note, however, that  the system $LVS$ is \emph{not}  a system of 
 Lotka-Volterra equations.

We now introduce the notation of coexisting traits in this context (cf.\ \cite{C_PES}).
\begin{definition} 
For any $d\geq 2$, we say that the distinct traits  $(g_1,p_1),\ldots ,(g_d, p_d)$ \emph{coexist} if the system $LVS(d,  (\mathbf g,\mathbf p))$ has a unique non-trivial equilibrium $\bar{\mathfrak n} (\mathbf g,\mathbf p)\in (0,\infty)^{|\mathcal X_{(\mathbf g,\mathbf p)}|}$ which is locally strictly stable,  meaning that all eigenvalues of the Jacobian matrix of the system $LVS(d,  (\mathbf g,\mathbf p))$ at $\bar{\mathfrak n} (\mathbf g,\mathbf p)$ have strictly negative real parts. 
\end{definition}

Note that if $(g_1,p_1),\ldots ,(g_d, p_d)$ coexist, then all traits of $\mathcal X_{(\mathbf g,\mathbf p)}$ coexist and the equilibrium $\bar{\mathfrak n} (\mathbf g,\mathbf p)$ is asymptotically stable.
We will prove later that if the traits $(g_1,p_1),\ldots ,(g_d, p_d)$ coexist, then the \emph{invasion probability} of a mutant trait $(\tilde g, \tilde p)$ which appears in the resident population $\mathcal X_{(\mathbf g,\mathbf p)}$ close to  $\bar{\mathfrak n} (\mathbf g,\mathbf p)$ is given by the function
\be\label{invasion probability} 1-q_{{(\mathbf g,\mathbf p)}}(\tilde g,\tilde p),
\ee
 where $q_{(\mathbf g,\mathbf p)}(\tilde g,\tilde p)$ is given as  follows: 
Let us denote the elements of $[\tilde p]_{\tilde g}$ by $\tilde p_1,\tilde p_2,\ldots,\tilde p_{ |[\tilde p]_{\tilde g}|}$ and assume without lost of generality that $\tilde p= \tilde p_1$. Then,  $q_{(\mathbf g,\mathbf p)}(\tilde g,\tilde p)$ is the first component of the smallest solution of 
\be\mathbf  u(\mathbf y)=\mathbf 0,\ee
 where $\mathbf  u$ is a map from $\mathbb R^{ |[\tilde p]_{\tilde g}|}$ to $\mathbb R^{ |[\tilde p]_{\tilde g}|} $ defined for all $i\in \{1,\ldots ,|[\tilde p]_{\tilde g}|\}$ by
\begin{align}
&u_i (\mathbf y)\equiv \\\nonumber 
	&\; b(\tilde p_i)\:y_i^2
	+\sum_{j=1}^{ |[\tilde p]_{\tilde g}|}
		 \biggl( s_{\text{nat.}}^{\tilde g}(\tilde p_i, \tilde p_j) +\!\!\sum_{(g, p)\in\mathcal X_{(\mathbf g,\mathbf p)}}
				 \!\!s^{\tilde g}_{\text{ind.}}(\tilde p_i, \tilde p_j)( p) \bar \frakn_{ ( g, p)}
		\biggr)\: y_j
	+d(\tilde p_i) 
	+ \sum_{(g,p)\in\mathcal X_{(\mathbf g,\mathbf p)}} 
		c(\tilde p_i, p)\bar{\mathfrak n}_{(g,p)}(\mathbf g,\mathbf p) \\\nonumber
	&-\:\Bigg( b(\tilde p_i)+\sum_{j=1}^{ |[\tilde p]_{\tilde g}|}
		 \biggl( s_{\text{nat.}}^{\tilde g}(\tilde p_i, \tilde p_j) +\!\!\sum_{(g, p)\in\mathcal X_{(\mathbf g,\mathbf p)}}
				 \!\!s^{\tilde g}_{\text{ind.}}(\tilde p_i, \tilde p_j)( p) \bar \frakn_{ (g, p)}
		\biggr)+d(\tilde p_i) +\!\! \sum_{(g,p)\in\mathcal X_{(\mathbf g,\mathbf p)}} \!\!c(\tilde p_i, p)\bar{\mathfrak n}_{(g,p)}(\mathbf g,\mathbf p) \Bigg)\:y_i.
\end{align}
In fact, $(1-q_{{(\mathbf g,\mathbf p)}}(\tilde g,\tilde p))$ is the probability that a single mutant survives in a resident population with traits $\mathcal X_{(\mathbf g,\mathbf p)}$. We obtain this by approximating the mutant population with multi-type branching processes (cf.\ proof of Thm.\ \ref{The three steps of invasion}). The function $(1-q_{(\mathbf g,\mathbf p)}(\tilde g,\tilde p))$ plays the same role as the function
$[f(y;\mathbf x)]_+ / b(y)$ in the standard case (cf.\  \cite{C_PES}).

To obtain that the process jumps on the evolutionary time scale from one equilibrium to the next, we need an assumption to prevent cycles, unstable equilibria or chaotic dynamics in the  
deterministic system (cf.\ \cite{C_PES} Ass.\ B).

\begin{assumption}\label{conv_to_fixedpoint}
For any given traits $(g_1,p_1),\ldots ,(g_d, p_d)\in \mathcal G\times \mathcal P$ 
that coexist and for any mutant trait $(\tilde g,\tilde p)\in\mathcal X\setminus \mathcal X_{(\mathbf g,\mathbf p)}$  
such that $ 1-q_{{(\mathbf g,\mathbf p)}}(\tilde g,\tilde p)>0$, there exists a neighbourhood $U\subset \mathbb R^{|\mathcal X_{(\mathbf g,\mathbf p)}|+ |[\tilde p]_{\tilde g}|}$ of  $(\bar{\mathfrak n}(\mathbf g,\mathbf p),0,\ldots,0)$
 such that all solutions of  $LVS(d+1,  ((\mathbf g,\mathbf p),(\tilde g,\tilde p)))$ with initial condition in $U\cap (0,\infty)^{|\mathcal X_{(\mathbf g,\mathbf p)}|+ |[\tilde p]_{\tilde g}|}$
converge as $t \uparrow  \infty$ to a unique locally strictly stable
 equilibrium in $\mathbb R^{|\mathcal X_{(\mathbf g,\mathbf p)}|+ |[\tilde p]_{\tilde g}|}$ denoted by $\mathfrak n^{*}((\mathbf g,\mathbf p),(\tilde g,\tilde p))$.
\end{assumption}

We write $\mathfrak n^{*}$ and not $\bar{\mathfrak n}$ to emphasise that 
some components of $\mathfrak n^{*}$ can be zero. We use the 
shorthand notation
$((\mathbf g,\mathbf p),(\tilde g,\tilde p))$ for $((g_1,p_1),\ldots ,(g_d, p_d),(\tilde g,\tilde p))$. 
Assumption \ref{conv_to_fixedpoint}  does not have to hold for all traits in $\mathcal X\setminus \mathcal X_{(\mathbf g,\mathbf p)}$, but  only for those traits $(\tilde g,\tilde p)$ which can appear in the resident population by mutation, i.e.\ only if $\sum_{(g,p)\in\mathcal X_{(\mathbf g,\mathbf p)}}m(g)M((g,p),(\tilde g,\tilde p))$ is positive.

\begin{remark}It is possible to extend the definitions and assumptions for the study of rare mutations and fast switches in populations with \emph{non-discrete trait} space if one assumes that an individual  can change its phenotype only to finitely many other phenotypes. This must be encoded in the switching kernels. More precisely, for all $(g,p)\in\mathcal G\times \mathcal P$  the communicating class $[p]_g$ should contain  finitely many elements. 
\end{remark}
%%%%%%%%%%%%%%%%%%%%%%%%%%%%%%%%%%%%%%%%%%%%%%%%%%%%%%
%%%%%%%%%%%%%%%%%%%%%%%%%%%%%%%%%%%%%%%%%%%%%%%%%%%%%%
\subsection{Convergence to the generalised Polymorphic Evolution Sequence. } \label{sec-Thm}
%%%%%%%%%%%%%%%%%%%%%%%%%%%%%%%%%%%%%%%%%%%%%%%%%%%%%%
%%%%%%%%%%%%%%%%%%%%%%%%%%%%%%%%%%%%%%%%%%%%%%%%%%%%%%
In this subsection we state the main theorem of this paper and give the general idea of the proof illustrated by an example. 

\begin{theorem}\label{PESwP} Suppose that the Assumptions \ref{switch}, \ref{recurrence} and  \ref{conv_to_fixedpoint} hold. 
Fix $(g_1,p_1),\ldots, (g_d,p_d)\in\mathcal G\times\mathcal P$ coexisting traits and assume that the initial conditions have support $\mathcal X_{(\mathbf g,\mathbf p)}$ and 
converge almost surely to $\bar{\mathfrak n} (\mathbf g,\mathbf p)$, i.e.\ $\lim_{K\uparrow \infty}  \nu_0^K=\sum_{x\in \mathcal X_{(\mathbf g,\mathbf p)}} \bar{\mathfrak n}_{x} (\mathbf g,\mathbf p)\delta_{x}$ a.s.. Furthermore, assume  that 
\begin{align}\label{Conv_Cond}
\forall V>0, \qquad \exp(-VK)\ll u_K \ll \frac{1}{K\ln(K)}, \qquad \text{as } K\uparrow \infty.
\end{align}
Then, the sequence of the rescaled processes $(\nu^{K}_{ t/ Ku_K})_{t\geq 0}$, generated by $ {\mathscr L}^K$ with initial state $\nu_0^K$, converges in the sense of finite dimensional distributions to the measure-valued pure jump process $\Lambda$ defined as follows:
$\Lambda_0=\sum_{(g,p)\in\mathcal X_{(\mathbf g,\mathbf p)}} \bar{\mathfrak n}_{(g,p)} (\mathbf g,\mathbf p) \delta_{(g,p)}$
and the process $\:\Lambda\:$ jumps for all $(\hat g,\hat p)\in \mathcal X_{(\mathbf g,\mathbf p)}$  from
 \be 
 \sum_{(g,p)\in\mathcal X_{(\mathbf g,\mathbf p)}} \bar{\mathfrak n}_{(g,p)} (\mathbf g,\mathbf p) \delta_{(g,p)}
\quad \text{ to } \quad
  \sum_{(g,p)\in\mathcal X_{((\mathbf g,\mathbf p),(\tilde g,\tilde p))}} \mathfrak n^*_{(g,p)} ((\mathbf g,\mathbf p),(\tilde g,\tilde p)) \delta_{(g,p)}
  \ee
  with infinitesimal rate
  \be
  m(\hat g)b(\hat p) \bar{\mathfrak n}_{(\hat g,\hat p)} (\mathbf g,\mathbf p) (1-q_{{(\mathbf g,\mathbf p)}}(\tilde g,\tilde p))M((\hat g,\hat p), (\tilde g,\tilde p)).
  \ee
\end{theorem}
 \begin{remark}
\begin{enumerate}[(i)]
\setlength{\itemsep}{0.5pt}
\item The convergence cannot  hold  in law for the Skorokhod topology (cf.\ \cite{C_TSS}). It holds only in the sense of finite dimensional distributions on $\mathcal M_F(\mathcal X)$, the set of finite positive measures on $\mathcal X$ equipped with the topology of the total variation norm.
\item The process $\Lambda$ is a generalised version of the usual PES. Therefore, we call $\Lambda$  Polymorphic Evolution Sequence with phenotypic Plasticity (PESP).
\item Assumption \ref{conv_to_fixedpoint} is essential for this statement.
In the case when the dynamical system 
has multiple attractors and different points near the initial state lie in different basins of attraction, it is not clear and may be random which attractor the system approaches. 
The characterisation of the asymptotic behaviour of the dynamical system
 is needed to describe the final
state of the stochastic process. This is in general a difficult
and complex problem, which is not doable analytically and requires numerical analysis. Thus, we restrict ourselves to the  Assumption \ref{conv_to_fixedpoint}.
\end{enumerate}
\end{remark}

We describe in the following the general idea of the proof, which is quite similar to the one given in \cite{C_PES}.
The population is either in a stable phase or in an invasion phase. Until the first mutant appears the population is in a stable phase, i.e.\ the  population stays close to a given equilibrium.
From the first mutational event until the population reaches again a stable state, the population is in an invasion phase. In fact, the mutant either survives and the population reaches fast  a new stable state (where the mutant trait is present) or the mutant goes extinct and the population is again in the old stable state. After this the populations is again in a stable phase until the next mutation, etc..

Note that we prove in the following that the invasion phases are relatively short $(O(\ln(K)))$ compared to the stable phase $(O(1/u_K K))$. Since we study the process on the time scale $1/Ku_K$, the limit process proceeds as a pure jump process which jumps from one stable state to another.
\\[0.5em]
\emph{The stable phase:} Fix $\epsilon> 0$.  Let $\mathcal X_{(\mathbf g,\mathbf p)}$ be the support of the initial conditions.
For large $K$, the population process 
 $\nu^K$ is, with high probability, still in a small neighbourhood of the equilibrium $\bar{\mathfrak n}(\mathbf g,\mathbf p)$ when the first mutant appears. 
In fact, using large deviation results on the problem of exit from a domain (cf.\ \cite{F_RPoDS}), we obtain that there exists a constant $M>0$ such that the first time  $\nu^K$  leave the $M\epsilon$-neighbourhood of $\bar{\mathfrak n}(\mathbf g,\mathbf p)$  is bigger than $\exp(V K)$ for some $V > 0$ with high probability.
Thus, until this stopping time, mutations born from individuals with trait $x\in\mathcal X_{(\mathbf g,\mathbf p)}$  appear with a rate which is close to 
\be u_K m(x)b(x) K\bar{\mathfrak n}_x(\mathbf g,\mathbf p).\nonumber \ee The condition $
 1/( {Ku_K})\ll \exp (V K)$ for all $V>0$ in  (\ref{Conv_Cond})
  ensures that the first mutation appears before this exit time.
\\[0.5em]  
\emph{The invasion phase:} We divide the invasion of a given mutant trait 
$(\tilde g,\tilde p)$ into \emph{three steps}, as in \cite{C_TSS} and 
\cite{C_PES} (cf.\ Fig.\ \ref{FigPESwP}).
In the first step, from a mutational event until the mutant population goes extinct or the mutant density reaches the value $\e$, the number of mutant individuals is
small (cf.\ Fig.\ \ref{FigPESwP}, $[0,t_1]$). Thus, applying a perturbed version of the large deviation result we used in the first phase, we obtain that
 the resident population stays close to its equilibrium density $\bar{\mathfrak n}(\mathbf g,\mathbf p)$ during this step. Using similar arguments as  Champagnat et al.\ \cite{C_TSS,C_PES}, we prove that
 the mutant population  is well approximated  by a $|[\tilde p]_{\tilde g}|$-type
branching process $Z$, as long as the mutant population 
has less than $\epsilon K$ individuals. More precisely, let  us denote the elements of $[\tilde p]_{\tilde g}$ by $\tilde p_1,\ldots,\tilde p_{|[\tilde p]_{\tilde g}|} $, then,
for each  $1\leq i\leq |[\tilde p]_{\tilde g}|$, each individual in $Z$ (carrying trait $(\tilde g,\tilde p_i)$) undergoes
\begin{enumerate}[(i)]
\setlength{\itemsep}{6pt}
 \item {birth (without mutation) with rate $b(\tilde p_i)$,}
\item {death with rate $ d(\tilde p_i)  + \sum_{(g,p)\in\mathcal X_{(\mathbf g,\mathbf p)}}c(\tilde p_i, p) 
\bar{\mathfrak n}_{(g,p)}(\mathbf g,\mathbf p)$ and}
\item{ switch to  $\tilde p_j$ with rate $s_{\text{nat.}}^{\tilde g}(\tilde p_i, \tilde p_j) +\!\!\sum_{( g, p)\in\mathcal X_{(\mathbf g,\mathbf p)}}
				 \!\!s^{\tilde g}_{\text{ind.}}(\tilde p_i, \tilde p_j)( p) \bar \frakn_{ ( g, p)}$ for all $1\leq j\leq |[\tilde p]_{\tilde g}|$.}
\end{enumerate} 
This continuous-time multi-type branching process is supercritical if and only if the largest eigenvalue of its infinitesimal generator, which we denote by $\lambda_{\text{max}}$,  is larger than zero. 
Hence, 
the mutant invades with positive probability if and only if $\lambda_{\text{max}}>0$. Moreover, the probability that the density of the mutant's genotype, $\nu^K(\tilde g)$, reaches $\epsilon$ at some time $t_1$ is close to the probability that the multi-type branching process reaches the total mass $\epsilon K$, which converges as $K\uparrow \infty$ to $(1-q_{{(\mathbf g,\mathbf p)}}(\tilde g,\tilde p))$. 

In the second step, we obtain as a consequence of  Theorem \ref{det-limit} that once the mutant density has reached $\epsilon$, for large $K$, the stochastic process $\nu^K$ can be approximated on any finite time interval by the solution of $LVS(d+1,  ((g_1,p_1),\ldots ,(g_d, p_d),(\tilde g,\tilde p)))$
with a given initial state. By Assumption \ref{conv_to_fixedpoint}, this solution reaches the $\epsilon$-neighbourhood of its new equilibrium $n^*((\mathbf g,\mathbf p),(\tilde g,\tilde p))$ in finite time.  Therefore, for large $K$, the stochastic process $\nu^K$ also reaches with high probability
the $\epsilon$-neighbourhood of $n^*((\mathbf g,\mathbf p),(\tilde g,\tilde p))$ at some finite ($K$ independent) time $t_2$.

In the third step, we  use similar arguments as in the first atep. Since $n^*((\mathbf g,\mathbf p),(\tilde g,\tilde p))$ is a
strongly locally stable equilibrium (Ass. \ref{conv_to_fixedpoint}), the 
stochastic process $\nu^K_t$ stays close $n^*((\mathbf g,\mathbf p),
(\tilde g,\tilde p))$ and we can approximate the densities of the traits $
(g,p)\in\mathcal X_{((\mathbf g,\mathbf p),(\tilde g ,\tilde p))}$
 with $n_{(g,p)}^*((\mathbf g,\mathbf p),(\tilde g,\tilde p)) = 0$ by $|[p]_g|$-
 type branching 
 processes which are subcritical and therefore become extinct, a.s.. 
 \\[0.5em]
The duration of the first and third step are 
proportional to $\ln(K)$, whereas 
the time of the second step is bounded.
Thus, the second inequality in (\ref{Conv_Cond}) guarantees that, with 
high probability, 
the three steps of invasion are completed before a new mutation occurs.
 After the last step the process is again back in a  stable phase, but with a 
 possibly different resident population, until the next mutation happens.
 
 \begin{figure}[h!]
	\centering
	 \includegraphics[width=.8\textwidth]{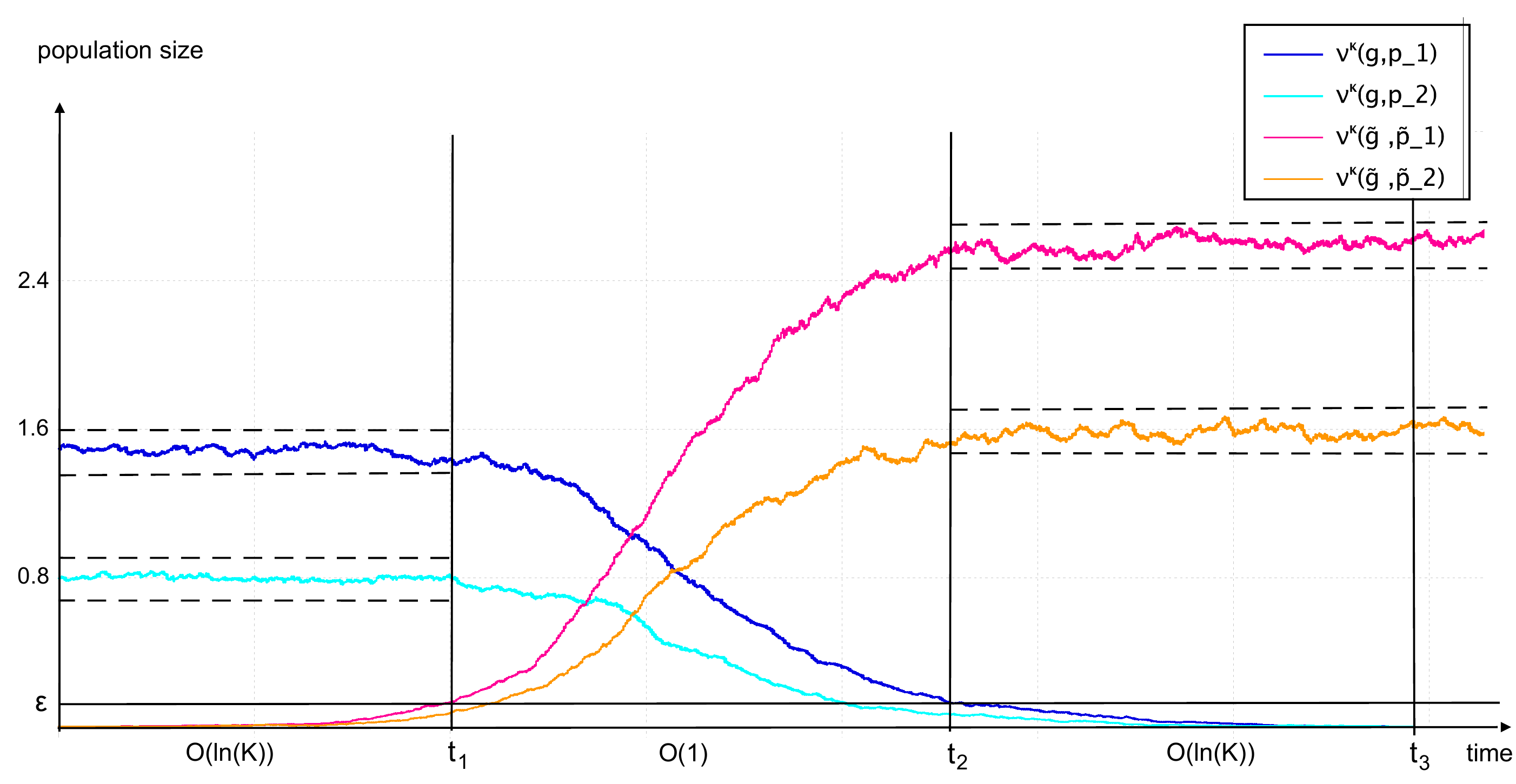}	
	\caption{\small {The three steps of one invasion phase.} }\label{FigPESwP}
\end{figure}
\normalsize
\textit{An example:} Figure \ref{FigPESwP} shows the invasion phase of  a single mutant with trait $(\tilde g, \tilde p_1)$, which appeared (at time $0$) in a population close to $\bar n(\mathbf g, \mathbf p)$ (indicated by the dashed lines).  In this example the resident population consists of two coexisting traits $(g,p_1)$ and $(g,p_2)$ and the mutant individuals can switch  to one other phenotype only, i.e.\  $[\tilde p_1]_{\tilde g}=\{\tilde p_1,\tilde p_2\}$. 
The parameters  of the simulation of Figure \ref{FigPESwP} are given in Table \ref{Table_PESwP}.
\begin{table}[h]
\centering\footnotesize
\begin{tabular}{|@{\hspace{ 0.15cm}}l@{\hspace{ 0.15cm}} |@{\hspace{ 0.15cm}}l @{\hspace{ 0.15cm}} |@{\hspace{ 0.1cm}}l@{\hspace{ 0.1cm}}|@{\hspace{ 0.1cm}}l@{\hspace{ 0.1cm}}| @{\hspace{ 0.1cm}}l @{\hspace{ 0.1cm}}|@{\hspace{ 0.1cm}}l @{\hspace{ 0.1cm}}|@{\hspace{ 0.1cm}}l@{\hspace{ 0.1cm}}|}\hline
$b(p_1)=3	  $ &$ d({p_1})=1	$ &$c({p_1},{p_1})=1	  $&$ c({p_1}, p_2)=0.7$&$ c({p_1},\tilde p_1)=0.7$&$ c(p_1,\tilde p_2)=0.7$&$s_{\text{nat.}}^{ g}( p_1,p_2)=1$ \\[0.1em]\hline
$b(p_2)=3	  $ &$ d({p_2})=1	$ &$c({p_2},{p_1})=0.7	  $&$ c({p_2}, p_2)=1$&$ c({p_2},\tilde p_1)=0.7$&$ c(p_2,\tilde p_2)=0.7$&$s^{ g}_{\text{nat.}}( p_2,p_1)=2$ \\[0.1em]\hline
$b(\tilde p_1)=5	  $ &$ d(\tilde p_1)=1	$ &$c(\tilde p_1,{p_1})=0.7	  $&$ c(\tilde p_1, p_2)=0.7$&$ c(\tilde p_1,\tilde p_1)=1$&$ c(\tilde p_1,\tilde p_2)=0.7$&$s^{\tilde g}_{\text{nat.}}(\tilde  p_1,\tilde p_2)=1.5$ \\[0.1em]\hline
$b(\tilde p_2)=4	  $ &$ d({\tilde p_2})=1	$ &$c({\tilde p_2},{p_1})=0.7	  $&$ c({\tilde p_2}, p_2)=0.7$&$ c({\tilde p_2},\tilde p_1)=0.7$&$ c(\tilde p_2,\tilde p_2)=1$&$s^{\tilde g}_{\text{nat.}}(\tilde  p_2,\tilde p_1)=2$ \\[0.1em]\hline\hline
$K=2000 $ &$u_K=0	$ &$\nu_0^K(g,p_1)=1.5 $& $\nu_0^K(g,p_2)=0.8 $& $\nu_0^K(\tilde g,\tilde p_1)=1/K $& $\nu_0^K(\tilde g,\tilde p_2)=0$&{$s_{\text{ind.}}^{.} (\: .\:,\:.\:)(.)\equiv 0 $}\\[0.1em]\hline
\end{tabular}
\caption{\small Parameters of Figure \ref{FigPESwP} }\label{Table_PESwP}
\end{table}
The stable fixed point of the system $LVS(2, ((g,p_1),(g,p_2)))$ is $\bar n ((g,p_1),(g,p_2))\approx 
(1.507, 0.809 )$. The infinitesimal generator  of the multi-type branching process that approximates the 
mutant population in the first step is approximately 
\be\lb\begin{matrix}
0.879 &1.5\\
2 & -0.621
\end{matrix}\rb.
\ee 
Since the largest eigenvalue of this matrix  is positive ($\approx 2.016$), the mutant population reaches with positive probability the second invasion step (cf.\ Fig.\  \ref{FigPESwP}). \\
Moreover, $n^*\approx (0,0, 2.608, 1.608)$ is  the unique locally strictly stable fixed point of the dynamical system $LVS(4, ((g,p_1),(g,p_2),(\tilde g,\tilde p_1),(\tilde g,\tilde p_2)))$.
The dynamical system and hence also the stochastic process reach in finite time the $\e$-neighbourhood of this value. 
The infinitesimal generator of the multi-type branching process that approximates the resident population in the third step is approximately 
\be\lb\begin{matrix}
-1.951 &2\\
1 & -2.951
\end{matrix}\rb.
\ee 
The largest eigenvalue of this matrix  is negative ($\approx -0.951$) meaning that the process is subcritical and goes extinct a.s.. Therefore, there exists a time $t_3$ such that all individuals which carry trait $(g,p_1)$ or $(g,p_2)$ are a.s.\ dead at time $t_3$. 
 %%%%%%%%%%%%%%%%%%%%%%%%%%%%%%%%%%%%%%%%%%%%%%%%%%%%%%
\subsection{Proof of Theorem \ref{PESwP}.}\label{PES-prove}
%%%%%%%%%%%%%%%%%%%%%%%%%%%%%%%%%%%%%%%%%%%%%%%%%%%%%%
 In this paragraph we prove the convergence to the PESP. (The proof uses the same arguments and techniques as \cite{C_PES}, which were developed in \cite{C_TSS}. However, some extensions are necessary, if fast phenotypic switches are included in the process, which we state and prove in this subsection.) We start with an analog  of Theorem 3 of \cite{C_TSS}.
Part (i) of the following theorem strengthens Theorem \ref{det-limit}, and part (ii) provides control  of exit from an attractive domain in the polymorphic case with phenotypic plasticity.

%%%%%%%%%%%%%%%%%%%%%%%%%%%%%%%%%%%%%%%%%%%%%%%%%%%%%%%%
\begin{theorem}\label{exit_domain_prob}
\begin{enumerate}[(i)]
\setlength{\itemsep}{0,5pt}
\item Assume that the initial conditions have support $\{(g_1,p_1),\ldots (g_d,p_d)\}$  and are uniformly bounded, i.e.\ , for all $1\leq i\leq d$, $\nu_0^K(g_i,p_i)\in A$, where $A$ is a compact subset of $\mathbb R_{>0}$.  
 Then, for all $T>0$
\be\label{LLN_UIC}
\lim_{K\uparrow \infty} \:\sup_{ t\in [0,T]}  
\:\Big|\Big|\nu^K_t-\sum_{x\in \mathcal X_{(\mathbf g,\mathbf p)}} \mathfrak n_x(t, \nu^K_0)\delta_{x}\Big|\Big|^{}_{TV}=0 \quad\text{a.s.,}
\ee
where $\mathfrak n(t, \nu^K_0)\in \mathbb R^{| \mathcal X_{(\mathbf g,\mathbf p)}|}$ denotes the value of the solution of $LVS(d,  (\mathbf g,\mathbf p))$  at time $t$ with initial condition $\mathfrak n_{ x}(0,\nu^K_0)=\nu^K_0( x)$ for all $x\in \mathcal X_{(\mathbf g,\mathbf p)}$.
Note that the measure $\sum_{x\in \mathcal X_{(\mathbf g,\mathbf p)}} \mathfrak n_x(t, \nu^K_0)\delta_{x}$ depends on $K$, since the initial condition and hence the solution of $LVS(d,  (\mathbf g,\mathbf p))$ depends on $K$.

\item Let $(g_1,p_1),\ldots, (g_d,p_d)\in\mathcal X$ coexist. Assume that, for any $K\geq 1$,
$ \text{Supp}(\nu_0^K)= \mathcal X_{(\mathbf g,\mathbf p)}$. Let $\tau^{}_{\text{mut.}}$ be the first mutation time. Define the first exit time from the
 $\xi$-neighbourhood of $\bar{\mathfrak n}_{x}(\mathbf g,\mathbf p)$ by
\be \theta^{K, \xi}_{\text{exit}}\equiv\inf\left\{t\geq 0: \exists x\in\mathcal X_{(\mathbf g,\mathbf p)} : \left|\nu_t^K(x)-\bar{\mathfrak n}_{x}(\mathbf g,\mathbf p)\right| >\xi\right\}. 
\ee
Then there exist $\epsilon_0>0$ and $M>0$ such that, for all $\epsilon<\epsilon_0$, there exists  $V>0$ such that if the initial state of $\nu^K$ lies in  the $\epsilon$-neighbourhood of $\bar{\mathfrak n}_{x}(\mathbf g,\mathbf p)$, the probability that
$ \theta^{K, M\epsilon}_{\text{exit}}$ is larger than $\eee^{KV}\wedge \tau^{}_{\text{mut.}}$  converges to one, i.e.\
\be\label{exit_domain}\quad
\lim_{K\uparrow  \infty}\: \sup_{\mathbf{n}^K\in(\mathbb N/K)^{|\mathcal X_{(\mathbf g,\mathbf p)}|}  \cap B_{\epsilon}(\bar{\mathfrak n}(\mathbf g,\mathbf p) )}
\mathbb P\left [  \theta^{K, M\epsilon}_{\text{exit}} \!<\eee^{KV}\wedge \tau^{}_{\text{mut.}}\: \Big|\: \nu^K_0(x)= n^K_x \text{ for all } x\in \mathcal X_{(\mathbf g,\mathbf p)}\right]=0,
\ee
where $\mathbf{n}^K\equiv (n^K_x)_{x\in\mathcal X_{(\mathbf g,\mathbf p)}}$ and $B_{\epsilon}(\bar{\mathfrak n}(\mathbf g,\mathbf p) )$ denotes the $\epsilon$-neighbourhood of $\bar{\mathfrak n}(\mathbf g,\mathbf p)$.

Moreover, (\ref{exit_domain}) also holds if, for all $(g,p)\in\mathcal X_{(\mathbf g,\mathbf p)}$, the total death rate of an individual with trait $(g,p)$, 
\be
d(p)+\sum_{(\tilde g,\tilde p)\in\mathcal X_{(\mathbf g,\mathbf p)}}c(p,\tilde p)\nu^K_t(\tilde g,\tilde p),
\ee
and the total switch rates of an individual with trait $(g,p)$, 
\be
s^g_{\text{nat.}}(p, p_i)+\sum_{(\tilde g,\tilde p)\in\mathcal X_{(\mathbf g,\mathbf p)}}s^g_{\text{ind.}}(p, p_i)(\tilde p)\nu^K_t(\tilde g,\tilde p)\quad \text{for all }p_i\in[p]_g,
\ee
are perturbed by additional random processes that are uniformly bounded by $\bar c \epsilon$ respectively $\bar s_{\text ind.}  \epsilon$, where  $\bar c $ and $ \bar s_{\text ind.} $ are upper bounds for the parameters of competition and induced switch.
\end{enumerate}
\end{theorem}

%%%%%%%%%%%%%%%%%%%%%%%%%%%%%%%%%%%%%%%%%%%%%%%%%%%%%%%%%%%%
\begin{remark}\begin{enumerate}[(i)]
\setlength{\itemsep}{0,5pt}
\item{
One consequence of the second part of (ii) is that, with high probability, the process stays in the $M\epsilon$-neighbourhood of $\bar{\mathfrak n}_{x}(\mathbf g,\mathbf p)$ until the first time that a mutant's density reaches the value $\epsilon$.
In other words, let $\theta^{K}_{\text{Invasion}}$ denote the first time that a mutant's density reaches the value $\epsilon$, i.e 
\be
\theta^{K}_{\text{Invasion}}\equiv \left\{\:t\geq 0: \exists (g,p)\notin \mathcal X_{(\mathbf g,\mathbf p)}: \textstyle\sum_{\tilde p \in [p]_g} \nu^K_t(g,\tilde p)\geq \epsilon \:\right\}.
\ee
Then, the probability that
$\theta^{K, M\epsilon}_{\text{exit}}$ is larger than $\eee^{KV}\wedge \theta^{K}_{\text{Invasion}}$  converges to one.
We use this result also for the third invasion step.}
\item{Since $\bar{\mathfrak n}(\mathbf g,\mathbf p) $ is a locally strictly stable fixed point of the  system $LVS(d,  (\mathbf g,\mathbf p))$, there exists a constant $M>0$ such that, for all $\epsilon>0$ small enough,  
for all trajectories $\mathfrak n(t)$ with 
$||\mathfrak n(0)  -\bar{\mathfrak n}(\mathbf g,\mathbf p)||<\epsilon$, it holds
that $\sup_{t\geq0}||\mathfrak n(t)-\bar{\mathfrak n}(\mathbf g,\mathbf p)|| < M \epsilon$. }
\end{enumerate}
\end{remark}

%%%%%%%%%%%%%%%%%%%%%%%%%%%%%%%%%%%%%%%%%%%%%%%%%%%%%%%%%%%%
\begin{proof}
The main task to prove (i) is to show that a large deviation principle on $[0,T]$ holds for a sightly modify process and that the $\nu^K$ has the same law on the random time interval we need to control it. 
In fact, Theorem 10.2.6 of \cite{D_LD} can be applied to obtain the large deviation principle.
The main task to prove (ii) is to show that  the classical estimates for  exit 
times  from a domain (cf.\ \cite{F_RPoDS}) for the jump process $\nu^K$ 
can be used. Note that Freidlin and Wentzell study in \cite{F_RPoDS} 
mainly small white noise perturbations of dynamical systems. However, 
there also are some comments on the generalisation to dynamical 
systems with small jump-like perturbations (cf.\ \cite{F_RPoDS}, Sec. 
5.4).
\end{proof}

The following Lemma describes the asymptotic behaviour of 
$ \tau^{}_{\text{mut.}}$ and can be seen as an extension of Lemma 2  of \cite{C_TSS} or Lemma A.3 of \cite{C_PES}. 
%%%%%%%%%%%%%%%%%%%%%%%%%%%%%%%%%%%%%%%%%%%%%%%%%%%%%%%%%%%%
\begin{lemma} Let $(g_1,p_1),..., (g_d,p_d)\in\mathcal X$ coexist. Assume that, for any $K\geq 1$,
$ \text{Supp}(\nu_0^K)= \mathcal X_{(\mathbf g,\mathbf p)}$. Let $\tau^{}_{\text{mut.}}$ denote the first mutation time. Then, there exists $\epsilon_0>0$ such that if the initial states of $\nu^K$ belong to the $\epsilon_0$-neighbourhood of $\bar{\mathfrak n}_{x}(\mathbf g,\mathbf p)$, then,  for all $\epsilon\in(0,\epsilon_0)$, 
\be\label{blablabla}
\lim_{K\uparrow  \infty}\:
\mathbb P\left [   \tau^{}_{\text{mut.}}>\ln (K), \sup_{ t\in[\ln (K),\tau^{}_{\text{mut.}} ]}\:
\:\Big|\Big|\nu^K_t-\textstyle\sum_{x\in \mathcal X_{(\mathbf g,\mathbf p)}} \bar {\mathfrak n}_x(\mathbf g,\mathbf p)\delta_{x}\Big|\Big|^{}_{TV} <\epsilon \right]=1,
\ee
Moreover, $(\tau^{}_{\text{mut.}}u_K K)_{K\geq 1}$ converges in law %as $K\uparrow  \infty$ 
to an exponential distributed random variable with parameter $\sum_{(g,p)\in \mathcal X_{(\mathbf g,\mathbf p)}}m(g)b(p)\bar {\mathfrak n}_{(g,p)}(\mathbf g,\mathbf p)$ and
the probability that the mutant, which appears at time $\tau^{}_{\text{mut.}}$, is born from an individual with trait $(g,p)\in \mathcal X_{(\mathbf g,\mathbf p)}$ converges  to
\be
\frac{m(g)b(p)\bar {\mathfrak n}_{(g,p)}(\mathbf g,\mathbf p)}
{\sum_{(\tilde g,\tilde p)\in \mathcal X_{(\mathbf g,\mathbf p)}}m(\tilde g)b(\tilde p)\bar {\mathfrak n}_{(\tilde g,\tilde p)}(\mathbf g,\mathbf p)}
\ee
as $K\uparrow  \infty$.
\end{lemma}

%%%%%%%%%%%%%%%%%%%%%%%%%%%%%%%%%%%%%%%%%%%%%%%%%%%%%%%%%%%%
\begin{proof}
There exist constants $C>0$ and $V>0$, such that on the time interval $[0,\exp(KV)]$ the total mass of the population, $\nu_t^K(\mathcal X)$, is bounded from above by $C$. 
 Therefore, we can construct an exponential random variable $A$ with parameter $C' K u_K$, where $C'=C \max_{g\in\mathcal G, p\in\mathcal P} m(g) b(p)$, such that 
 \be
  A\leq \tau^{}_{\text{mut.}}\qquad
 \text { on the event }\quad\left\{\tau^{}_{\text{mut.}}< \exp(KV)\right\}.
 \ee Thus, $\mathbb P\left[\tau^{}_{\text{mut.}}>\ln(K)\right]\geq \mathbb P\left[A>\ln(K)\right]=\eee^{-C' \ln(K)K u_K}$. Since $(\ref{Conv_Cond})$ implies that $\ln(K)Ku_K $ converges to zero as $K\uparrow \infty$, we get $\lim_{K\uparrow \infty}\mathbb P[\tau^{}_{\text{mut.}}>\ln(K)]=1$.
 
  The fixed point $\bar {\mathfrak n}(\mathbf g,\mathbf p)$ is asymptotic stable. Thus,
$\exists \epsilon_0 > 0: \forall \tilde{\epsilon}\in(0,\epsilon_0)\:\exists T(\tilde{\epsilon})$:
\be
\| \mathfrak n(\mathbf g,\mathbf p)(0) - \bar {\mathfrak n}(\mathbf g,\mathbf p)\|<\epsilon_0, \qquad \text { implies } \quad
\sup_{ t\geq T(\tilde{\epsilon})}|\mathfrak n(\mathbf g,\mathbf p)(t)  - \bar {\mathfrak n}(\mathbf g,\mathbf p)|<\tilde{\epsilon}/2. \ee
In words, there exists a finite time $T(\tilde{\epsilon})$ such that all trajectories, which start in the $\epsilon_0$ neighbourhood of the fixed point, stay after $T(\tilde{\epsilon})$ in the  $\tilde{\epsilon}/2$-neighbourhood of the fixed point. 

Next, we apply the last theorem: By (i), for all $\tilde{\epsilon}\in(0,\epsilon_0) \:\exists T(\tilde{\epsilon})$ such that, for $K$ large enough,  
\be
\label{tilde_epsilon} \Big|\Big|\nu^K_{T(\tilde{\epsilon})}- \textstyle\sum_{x\in \mathcal X_{(\mathbf g,\mathbf p)}} \bar {\mathfrak n}_x(\mathbf g,\mathbf p)\delta_{x}\Big|\Big|^{}_{TV}<\tilde{\epsilon} \quad a.s..
\ee
 Then, by (ii), there exist $ \e_0>0$ and $ M>0$: for all $\tilde\epsilon\in(0,\epsilon_0)$ there exists  $V>0$ such that
 \be
 \lim_{K\uparrow \infty} \mathbb P\left[\sup_{t\in[ T(\tilde \e),\: \eee^{KV}\!\wedge \tau^{}_{\text{mut.}})}\:\Big|\Big|\nu^K_t-\textstyle\sum_{x\in \mathcal X_{(\mathbf g,\mathbf p)}} \bar {\mathfrak n}_x(\mathbf g,\mathbf p)\delta_{x}\Big|\Big|^{}_{TV}<M\tilde \epsilon\right]=1.
 \ee
Moreover,  for all $\tilde\epsilon\in(0,\epsilon_0)$  there exists $K_0\in \mathbb N$ such that $T(\tilde\epsilon)<\ln(K)$ for all $K\geq K_0$. Thus, setting $\epsilon=M\tilde \e$, ends the proof of (\ref{blablabla}), provided that $\lim_{K\uparrow \infty}\P[\tau^{}_{\text{mut.}}<\eee^{KV} ]=1$.

Again,  we can construct for all $\epsilon>0$ two exponential random variables $A^{1,K,\epsilon}$ and $A^{2,K,\epsilon}$ with parameters 
\be
a_1u_K K\equiv\sum_{(g,p)\in \mathcal X_{(\mathbf g,\mathbf p)}}u_K m(g)b(p)(\bar {\mathfrak n}_{(g,p)}(\mathbf g,\mathbf p)+\epsilon)K 
\ee
and
\be a_2u_K K\equiv\sum_{(g,p)\in \mathcal X_{(\mathbf g,\mathbf p)}}u_Km(g)b(p)(\bar {\mathfrak n}_{(g,p)}(\mathbf g,\mathbf p)-\epsilon)K
\ee 
such that 
\be 
A^{1,K,\epsilon}\leq \tau^{}_{\text{mut.}}\leq A^{2,K,\epsilon}\quad
\qquad \text{on the event } \{T(\tilde\epsilon)<\tau^{}_{\text{mut.}}<\eee^{KV}\},
 \ee 
  where $T(\tilde\epsilon)$ is the time defined  in equation (\ref{tilde_epsilon}) and $\tilde\e=\e/M$.
Moreover, we have 
 \be\lim_{K\uparrow \infty} \mathbb P[\tau^{}_{\text{mut.}}< \ln(K)]=0\qquad \text{and}\qquad
  \lim_{K\uparrow \infty} \mathbb P[A^{2,K,\epsilon}>\eee^{KV}]=0,
 \ee 
 because  $u_K K \eee^{KV}\uparrow \infty$ as $K\uparrow \infty$. Therefore, for all $\epsilon>0$,
 the probability of the event $\{T(\tilde\epsilon)<\tau^{}_{\text{mut.}}< \eee^{KV}\}$ converges to one as $K$ goes to infinity. Moreover, the random variables $A^{1,K,\epsilon}u_K K$ and $A^{2,K,\epsilon}u_K K$ converge both in law to the same exponential distributed random variable with parameter 
 \be
 \sum_{(g,p)\in \mathcal X_{(\mathbf g,\mathbf p)}}m(g)b(p)\bar {\mathfrak n}_{(g,p)}(\mathbf g,\mathbf p)
 \ee as first $K\uparrow  \infty$ and then $\epsilon\to0$.
The random variables  $A$, $A^{1,K,\epsilon}$ and $A^{2,K,\epsilon}$ can easily be constructed by using the pathwise description of $\nu^K$  (cf.\ \cite{BBC2015} or \cite{ChaJabMel}).
\end{proof}
%%%%%%%%%%%%%%%%%%%%%%%%%%%%%%%%%%%%%%%%%%%%%%%%%%%%%%%%%%%%
\begin{theorem}[The three steps of invasion]\label{The three steps of invasion}
 Let $(g_1,p_1),\ldots, (g_d,p_d)\in\mathcal X$ coexist. Assume that, for any $K\geq 1$,
$ \text{Supp}(\nu_0^K)= \mathcal X_{(\mathbf g,\mathbf p)}\cup\{(\tilde g,\tilde p)\}$. Let $\tau^{}_{\text{mut.}}$ denote the next mutation time (after time zero) and define
\bea\theta^{K,\xi}_{\text{No Jump}}&\equiv&
\inf\Big\{t\geq 0: \nu^K_t(\tilde g)=0 \;\text{ and }  \; \Big|\Big|\nu_t^K-\textstyle\sum_{ x\in\mathcal X_{(\mathbf g,\mathbf p)}}\bar{\mathfrak n}_{x}(\mathbf g,\mathbf p)\delta_{x}\Big|\Big|_{TV} <\xi\Big\}\\
 \theta^{K,\xi}_{\text{Jump}}&\equiv&
\inf\Big\{t\geq 0: \Big|\Big|\nu_t^K-\textstyle\sum_{x\in\mathcal X_{((\mathbf g,\mathbf p),(\tilde g,\tilde p))}} \mathfrak n^*_{x} ((\mathbf g,\mathbf p),(\tilde g,\tilde p)) \delta_{x}\Big|\Big|_{TV} <\xi
\\ && \nonumber \hspace{2cm} \text{ and }\;\forall \hat x
\notin\{x\in\mathcal X: \mathfrak n^*_{x} ((\mathbf g,\mathbf p),(\tilde g,\tilde p))>0\} :\: \nu^K_t(\hat x)=0 \Big\}.
\eea
Assume that we have a single initial mutant, i.e.\ $\nu^K_0(\tilde g,\tilde p)=1/K$. Then, there exist $\epsilon_0>0, C>0,$  and $M>0$ such that 
for all $\epsilon\in(0,\epsilon_0)$ if $||\nu_0^K-\textstyle\sum_{ x\in\mathcal X_{(\mathbf g,\mathbf p)}}\bar{\mathfrak n}_{x}(\mathbf g,\mathbf p)\delta_{x}||_{TV} <\epsilon$,
\bea
\label{Jump_pro}\lim_{K\uparrow \infty}\mathbb P\left[ \theta^{K, M\epsilon}_{\text{No Jump}}< \theta^{K, M \epsilon}_{\text{Jump}}\right]&\geq& q_{{(\mathbf g,\mathbf p)}}(\tilde g,\tilde p)-C\e,\\
\lim_{K\uparrow \infty}\mathbb P\left[ \theta^{K, M\epsilon}_{\text{Jump}}< \theta^{K, M\epsilon}_{\text{No Jump}}\right]&\geq& 1-q_{{(\mathbf g,\mathbf p)}}(\tilde g,\tilde p)- C \e,
\eea
where $1- q_{{(\mathbf g,\mathbf p)}}(\tilde g,\tilde p)$ is the invasion probability defined in (\ref{invasion probability}) and 
\be
\forall \eta>0,\quad \lim_{K\uparrow \infty}\mathbb P\left[ \theta^{K,M\epsilon}_{\text{Jump}}\wedge\theta^{K,M\epsilon}_{\text{No Jump}}\geq \frac{\eta}{u_K K}\wedge\tau^{}_{\text{mut.}} \right]\leq C\e.
\ee
\end{theorem}

The structure of the proof is similar to the one of Lemma 3 in \cite{C_TSS} (cf.\ also Lem.\  A.4.\ of \cite{C_PES}). However, we have to extend  the theory to multi-type
branching processes. Thus, the proof is not a simple copy the arguments in \cite{C_TSS}. Before 
proving the theorem, let us collect \emph{some properties about multi-type continuous-time branching 
processes}. Most of these can be found  in \cite{A_BP} or \cite{SEW}. The limit 
theorems we need in the sequel were first obtained  by Kesten and Stigum 
\cite{KesSte1,KesSte2,KesSte3} in the discrete-time case and by Athreya \cite{A_MTBP} in the 
continuous-time case. 

Let $Z(t)$ be a  $k$-dimensional continuous-time branching process. Assume that $Z(t)$ is 
non-singular and that the first moments exist. (Note that a process is \emph{singular} if and only if each individual has exactly one offspring and that the existence of the first moments is sufficient for the non-exposition hypothesis.)   
Then, the  so-called \emph{mean matrix}  $M(t)$ of  $Z(t)$ is the $k\times k$ matrix with elements 
\be
 m_{ij}(t)\equiv \mathbb E[Z_j(t)|Z(0)=\mathbf e_i ], \quad 1\leq i,j\leq k,
\ee
and $\mathbf e_i$ is the $i$-th unit vector in $\mathbb R^k$.
It is well known (cf.\  \cite{A_BP} p.\ 202) that there exists a matrix $\mathbf A$, called the infinitesimal generator of the semigroup $\{\mathbf M(t), t\geq 0\}$, such that 
\be\label{exp_dar}
\mathbf M(t)\equiv \exp(\mathbf A t)=\sum_{n=0}^{\infty}\frac{t^n(\mathbf A)^n}{n!}.
\ee
Furthermore, let $\mathbf r=(r_1,\ldots, r_k)$ be the vector of the branching rates, meaning that every individual
of type $i$ has an exponentially distributed lifetime of parameter $r_i$ and let  $\mathbf M$ be the mean 
matrix of the corresponding discrete-time process, i.e.\ $\mathbf M\equiv \{m_{ij} ,\; i,j=1,\ldots, k\}$,  where $m_{ij}$ is the expected number of type $j$ offspring of a single type-$i$-particle in one generation.  
Then, we can identify the infinitesimal generator $\mathbf A $ as
\be\label{inf. gen.}
\mathbf A= \mathbf R ( \mathbf M- \mathbf I),
\ee  
where  $\mathbf R = diag(r_1,\ldots, r_k)$, i.e.\ $r_{ij}= r_i\delta_i(j)$ and $\mathbf I$ is the identity matrix of size $k$.

Under the basic assumption of \emph{positive regularity}, i.e.  that there exists a time $t_0$ such that $\mathbf M(t_0)$ has strictly  positive entries, the Perron-Frobenius theory asserts that

\begin{enumerate}[(i)]
\setlength{\itemsep}{6pt}
\item the largest eigenvalue of $\mathbf M(t_0)$ is real-valued and  strictly positive,
\item the algebraic and geometric multiplicities of this eigenvalue are both one, and 
\item the corresponding  eigenvector has  strictly positive cmponents.
\end{enumerate}
 By (\ref{exp_dar}), the eigenvalues of $\mathbf M(t)$ are given by $
 \exp({\lambda_i t})$, where 
  $\{\l_i; i=1,\ldots, k\}$ are the eigenvalues of  $\mathbf A$, and both 
  matrices have the same eigenvectors, which implies that the left and right 
  eigenvectors  
 $\mathbf u$ and $\mathbf v$ of  $\l_{\text{max}}(\mathbf A)$ can be chosen with strictly positive 
 components and satisfying 
 \be
\textstyle \sum_{i=1}^k v_iu_i=1\quad \text{and} \quad\sum_{i=1}^k u_i=1.
\ee
 The process $Z$ is called supercritical, critical, or subcritical according as $\l_{\text{max}}(\mathbf A)$ is larger, equal, or smaller than zero. 
 
Observe that the following properties are equivalent  (cf.\  \cite{SEW} p. 95-99 and \cite{DISS_CTMTBP}):\\[0.5em]
$Z$ is irreducible  \; $\Leftrightarrow$\; $\mathbf M$ is irreducible \; $\Leftrightarrow$\; 
 $\mathbf A$ is irreducible \; $\Leftrightarrow$\; 
$\mathbf M(t)$ is irreducible for all $t > 0$ \; $\Leftrightarrow$\; 
 $\mathbf M(t)>0$ for all $t>0$.\\[0.5em]
In particular, irreducible implies positive regular. 
Note that a matrix is irreducible if it is not similar via a permutation to a block upper triangular matrix and that a Markov chain is irreducible if and only if the mean matrix is irreducible.

The following lemma is an extension of Theorem 4 of \cite{C_TSS} for 
multi-type branching processes.
%%%%%%%%%%%%%%%%%%%%
\begin{lemma}\label{expansion of Thm 4}
Let $(Z(t))_{t\geq0}$ be a non-singular and irreducible $k$-dimensional 
continuous-time Markov branching process and $\mathbf q$ the 
extinction vector of $Z$, i.e.
\be q_i\equiv \mathbb P[Z(t)=0 \text{ for some }t\geq0 | Z(0)=\mathbf e_i] 
\quad \text{ for $1\leq i\leq k$}.\ee Furthermore, let $(t_K)_{K\geq 1}$ be 
a sequence of positive numbers such that $\ln(K)\ll t_K$, define $T_{\rho}
\equiv\inf \{t\geq 0: \sum_{i=1}^k Z_i(t)=\rho\}$ and
assume that, for all $i,j\in\{1,\ldots, k\}$ and $t\in[0,\infty)$,  
\be
\mathbb E[Z_j(t) \ln(Z_j(t))| Z(0)=\mathbf e_i]<\infty.
\ee
 \begin{enumerate}[(i)]
 \item{If $Z$  is \emph{subcritical}, i.e.\ $\l_{\text{max}}(\mathbf A)<0$, 
 then for any $\epsilon>0$
 \be
 \lim_{K\uparrow \infty}\mathbb P\left[T^{}_0\leq t_K\wedge T^{}_{\lceil \e 
 K\rceil }\:\middle |\: Z (0)=\mathbf e_i \right ]=1 \qquad \text{ for all }i \in 
 \{1, \ldots, k\}
 \ee
 and 
  \be\label{(i)2}
 \lim_{K\uparrow \infty}\,\inf_{ \mathbf x\in \partial B_{\e K}  }\mathbb P\left[T^{}_0\leq t_K\,\middle|\, Z (0)= \mathbf x\right]=1,\quad\text { where }  \partial B_{\e K}\equiv\{ \mathbf x\in \mathbb N_0^k:  \textstyle\sum_{i=1}^{k} x_i =\lceil \e K\rceil \}.
 \ee
Moreover, for 
$\bar u=  \frac{\max_{1\leq i\leq k}u_i}{\min_{1\leq j\leq k}u_j}$ and for any $\e>0$,
\be\label{(i)3}
\quad\lim_{K\uparrow \infty }
\sup_{\mathbf x \in B_{\e^2 K} }\mathbb P\left[ T^{}_{\lceil \e K\rceil }\leq T^{}_0\:\middle|\:  Z (0)= \mathbf x\right]\leq\bar u \e,\quad \text {where }  B_{\e^2 K}\equiv\{ \mathbf x\in \mathbb N_0^k: \textstyle\sum_{i=1}^{k} x_i \leq\lceil \e^2K\rceil \}.\ee
}
  \item{If $Z$  is \emph{supercritical}, i.e.\ $\l_{\text{max}}(\mathbf A)>0$, then for any $\epsilon>0$ (small enough)
 \be\label{(ii)1}
 \lim_{K\uparrow \infty}\mathbb P\left[T^{}_0\leq t_K\wedge T^{}_{\lceil \e K\rceil }\:\middle|\: Z (0)=\mathbf e_i \right]=q_i \qquad \text{ for all }i \in \{1, \ldots, k\}
 \ee
 and 
  \be\label{(ii)2}
 \lim_{K\uparrow \infty}\mathbb P\left[T^{}_{\lceil \e K\rceil }\leq t_K\:\middle|\: Z (0)=\mathbf e_i \right]=1-q_i\qquad\text{ for all }i \in \{1, \ldots, k\}.
 \ee
Moreover, conditionally on survival, the proportions of the different types present in the population converge
almost surely,  as $t \uparrow  \infty$, to the corresponding ratios of the components
of the eigenvector:  for  all $ i=1,\ldots,k,$
\begin{equation}
\label{limit-ratio}
\lim_{t \uparrow  \infty}\frac{Z_{i}(t)}{\sum_{j=1}^{k}Z_j(t)}=\frac{v_{i}}{\sum_{j=1}^{k}v_j},
\quad \text{a.s. on } \{T_0=\infty\}.
\end{equation}
 }\end{enumerate}
\end{lemma}
%%%%%%%%%%%%%%%%%%%%%%%%%%%%%%%%%%%%%%%%%5
\begin{proof}
%%%%%%%%%%%%%%%%%%%%%%%%%%%%%%[Proof of Lemma \ref{expansion of Thm 4}
We start with the proof of
(i).  Since $Z(t)$ is in this case a subcritical irreducible continuous-time branching process and $\mathbb E[Z_j(t) \ln(Z_j(t))| Z(0)=\mathbf e_i]<\infty$, we obtain by applying Satz 6.2.7 of \cite{SEW} the existence of a constant $C>0$ such that
 \be\label{Lim_Extinct}
 \lim_{t \uparrow  \infty} \frac{1-q_i(t)}{\eee^{\l_{\text{max}}(\mathbf A) t}}=C  u_i,
 \ee 
 where $q_i(t)\equiv\mathbb P[Z(t)=0 \: |\: Z(0)=\mathbf e_i]$.  Moreover, we have a non-explosion condition. Thus, for all $\epsilon>0$, either $T^{}_{\lceil \e K\rceil }$ equals infinity or it converges to infinity as $K \uparrow \infty$. Putting both together, there exists a sequence $s_K$ with $\lim_{K \uparrow  \infty} s_K=+ \infty$ such that  
 \be
 \lim_{K\uparrow \infty}\mathbb P\left[T^{}_0\leq t_K\wedge T^{}_{\lceil \e K\rceil }\middle | Z (0)=\mathbf e_i \right ]\geq 
  \lim_{K\uparrow \infty}\mathbb P\left[T^{}_0\leq s_K\:\middle |\: Z (0)=\mathbf e_i \right ]= \lim_{K\uparrow \infty} q_i(s_K)=1.
 \ee
The branching
property implies that  for all $\mathbf x\in \mathbb N^k$, $\mathbb P[Z(t) = 0\:|\: Z(0)= \mathbf x] = \prod_{i=1}^{k}(q_i(t))^{x_i}$ (cf.\ \cite{DISS_CTMTBP} p.\ 25).
So, we get 
  \be\label{for (i) 2} 
  \inf_{\mathbf x\in\partial B_{\e K} }\!\mathbb P\left[T^{}_0\leq t_K\:\middle |\: Z (0)=\mathbf x \right ]= 
  \inf_{\mathbf x\in \partial B_{\e K} }\!\mathbb P\left[Z(t_K)=0\:\middle |\: Z (0)=\mathbf x \right ]=
  \inf_{\mathbf x\in \partial B_{\e K} } \prod_{i=1}^{k}(q_i(t_K))^{x_i}.
\ee
For all $i\in\{1,\ldots, k\}$, $1\geq (q_i(t_K))^{x_i}\geq(q_i(t_K))^{\lceil\e K\rceil} $ and by (\ref{Lim_Extinct}) we have $1-q_i(t_K)=O(\eee^{{\l_{\text{max}}(\mathbf A) t_K}})$. Moreover,  
for any sequence $ (w_K)_{K\geq 1}$ such that  $\lim_{K\uparrow \infty}w_K=0$, 
\be 
\lim_{K\uparrow \infty}\lb1+\frac{w_K}{K}\rb^K=1.
\ee 
This implies that, for all $t_K$ with $t_K\gg\ln(K)$ and $C>0$, since  $
\lim_{K\uparrow \infty} C \eee^{\l_{\text{max}}(\mathbf A) t_k}\lceil\e 
K\rceil=0$,
\be\lim_{K\uparrow \infty }(1-C\eee^{\l_{\text{max}}(\mathbf A) 
t_k})^{\lceil\e K\rceil}=1.\ee 
Thus, taking the limit $K\uparrow  \infty$ in (\ref{for (i) 2}), we obtain the 
desired equation (\ref{(i)2}). To prove the  inequality (\ref{(i)3}) we use the 
fact that 
$(\sum_{i=1}^k u_iZ_i(t))\eee^{-\lambda_{\text{max}}t}$ is a martingale (cf.\ \cite{A_MTBP}, Prop.\  2).  
By applying Doob's stopping theorem to the stopping time $T_{\lceil\e K\rceil}\wedge T_0$ we obtain, for all $\mathbf x\in B_{\e^2 K}$, that
\be\textstyle
\mathbb E\left[\left(\sum_{i=1}^k u_iZ_i(T_{\lceil\e K\rceil})\right)\eee^{-\lambda_{\text{max}}(\mathbf A)T_{\lceil\e K\rceil}}\mathds 1_{\{T_{\lceil\e K\rceil}< T_0\}}^{}\middle| Z(0)=\mathbf x\right]=\sum_{i=1}^ku_i x_i.
\ee
Therefore, since $\lambda_{\text{max}}(\mathbf A)<0$ in the subcritical case, 
\be
\mathbb E\left[\min_{1\leq i\leq k} u_i\lceil\e K\rceil \mathds 1_{\{T_{\lceil\e K\rceil}< T_0\}}^{}\middle| Z(0)=\mathbf x\right]\leq\max_{1\leq i\leq k} u_i\lceil \e^2 K\rceil, \quad \text{  for all $\mathbf x\in B_{\e^2 K}$},
\ee
which implies (\ref{(i)3}).

Let us continue by proving
(ii). Since $Z(t)$ is supercritical in this case,
applying Theorem 5.7.2 of \cite{A_BP} yields that
\be\label{Thm 5.7.2} 
\lim_{t \uparrow  \infty} Z(t)(\o) \eee^{-\l_{\text{max}}(\mathbf A) t}=W(\omega) \mathbf v, \hbox{\rm a.s.}, 
\ee
 where $W$ is a nonnegative random variable. Since we
assume that,
for all $i\in\{1,\ldots, k\}$, \\
 $\mathbb E[Z_j(t) \ln(Z_j(t))| Z(0)=\mathbf e_i]<\infty$, we get that 
 \be
\mathbb P[W=0| Z(0)=\mathbf e_i ]=q_i, \quad \mathbb E[W| Z(0)=\mathbf e_i ]=u_i,
\ee
and $W$ has an absolutely continuous distribution on $(0,\infty)$.
 All components of $\mathbf v$ are strictly positive and $W>0$, a.s., on 
 the event $\{\o:T_0(\o)=\infty\}$. Hence, we have 
\be 
Z(t)=O\lb \eee^{\l_{\text{max}}(\mathbf A) t}\rb \quad\text{ a.s.\quad  on }
\{T_0=\infty\}.
\ee
This implies, for $K$ large enough, $\mathbb P[ Z(t_K)< \lceil \e K\rceil, 
T_0=\infty]=0$ and thus
\be
\lim_{K\uparrow \infty}\mathbb P[T_0=\infty,\: T_{\lceil \e K\rceil }\geq 
t_K]=0.
\ee
Note that we used that $t_K\gg\ln(K)$.
Since $\mathbb P\left[T_0=\infty| Z(0)=\mathbf e_i \right]=1-q_i $, we deduce (\ref{(ii)2}).
On the other hand, there exist two sequences $s^1_K$ and $s^2_K$, which  converge to infinity as $K\uparrow  \infty$,  such that, for $K$ large enough,
 $s_K^1\leq t_K\wedge T_{\lceil \e K\rceil }\leq s_K^2$ a.s.. This implies (\ref{(ii)1}), because for all $i\in\{1,\ldots k\}$ and $l=1,2$, hold 
$\lim_{K\uparrow \infty}\mathbb P[T_0<s_K^l|Z_0=\mathbf e_i]=q_i$. 
Note that equation (\ref{limit-ratio}) is a simple consequence  of (\ref{Thm 5.7.2}).
\end{proof}
Using these properties about multi-type branching processes we can now prove the theorem about the three steps of invasion.
\begin{proof}[Proof of Theorem \ref{The three steps of invasion}]\emph{The first invasion step.} Let us introduce the following stopping times
\bea
\theta^{K, M\epsilon}_{\text{exit}}&=&\inf\Big\{t\geq 0: ||\nu_t^K-\textstyle\sum_{ x\in\mathcal X_{(\mathbf g,\mathbf p)}}\bar{\mathfrak n}_{x}(\mathbf g,\mathbf p)\delta_{x}||_{TV} >M \epsilon\Big\}\\
\tilde \theta^{K}_\epsilon&=&\inf\Big\{t\geq 0: \:\nu_t^K(\tilde g) \geq \epsilon\:\Big\}\\
\tilde\theta^{K}_{0}&=&\inf\Big\{t\geq 0: \:\nu_t^K(\tilde g)= 0\:\Big\}
\eea
Until $\tilde\theta^{K}_{\epsilon}$ the mutant population $\nu_t^K(\tilde g)$ influences only the death and switching rates of the resident population and this perturbation is uniformly bounded by $(\bar c + \bar s_{\text{ind.}})\epsilon$. Thus, by applying Theorem \ref{exit_domain_prob} (ii), we obtain
\be\label{exit_domain_apply}
\lim_{K\uparrow  \infty}
\mathbb P\left [  \theta^{K, M \epsilon}_{\text{exit}} <\eee^{KV}\!\wedge \tau^{}_{\text{mut.}}\wedge \tilde\theta^{K}_{\epsilon}\right]=0.
\ee
On the time interval  $[0,\theta^{K,M \epsilon}_{\text{exit}} \wedge \tau^{}_{\text{mut.}}\wedge \tilde\theta^{K}_{\epsilon}]$, the resident population can be approximated by $\textstyle\sum_{ x\in\mathcal X_{(\mathbf g,\mathbf p)}}\bar{\mathfrak n}_{x}(\mathbf g,\mathbf p)\delta_{x}$ and no further mutant appears. This allows us to approximate $\nu_t^K(\tilde g)$ by multi-type branching processes.

Let $k\equiv |[\tilde p]_{\tilde g}|$. We construct two $(\mathbb N_0)^{k}$- valued processes $X^{1,\e}(t)$ and $X^{2,\e}(t)$, using the pathwise definition in terms of Poisson point measures of $\nu^K_t$, 
which control the mutant population $\nu_t^K(\tilde g)$.
To this aim let us denote the elements of $[\tilde p]_{\tilde g}$ by $\tilde p_1,\ldots,\tilde p_{k}$ (w.l.o.g.\ $\tilde p\equiv\tilde p_1$). Then, we   define $X^{1,\e}$ by
\begin{align}
 X^{1,\e}(t)
&\equiv  X^{1,\e}(0)+\sum_{j=1}^{k}\int_0^t \!\int_{\mathbb N_0}\!\int_{\mathbb R_+}
	\mathds 1_{\left\{i\leq X^{1,\e}_{j}({s-}),\; 	
				\theta \leq b(\tilde p_i)-\epsilon\right\}}
	  \mathbf e_j N^{\text{birth}}_{(\tilde g,\tilde p_j)}(ds,di,d\th)\\\nonumber
&\quad
 -\sum_{j=1}^{k}\int_0^t \!\int_{\mathbb N_0}\!\int_{\mathbb R_+}
	\mathds 1_{\left\{i\leq X^{1,\e}_{j}({s-}),\; 
					\theta \leq d (\tilde p_j)
						+\sum_{(g,p)\in \mathcal X_{(\mathbf g,\mathbf p)}}c( \tilde p_j, p )  \bar{\mathfrak n}_{(g,p)}(\mathbf g,\mathbf p)+\bar c M\epsilon \right\}}
		 \mathbf e_j  N_{(\tilde g,\tilde p_j)}^{\text{death}}(ds,di,d\th)	\\	\nonumber
&\quad+	\sum_{j=1}^{k}\int_0^t \!\int_{\mathbb N_0}\!\int_{\mathbb R_+}\!\int_{[\tilde p]_{\tilde g}}
	\mathds 1_{\left\{i\leq X^{1,\e}_{j}({s-}),\:i\neq j \right\}} \:
	\bigg( \:\mathds 1_{\left\{ \theta \leq s_{\text{nat.}}^{\tilde g} (\tilde p_j,\tilde p_l)
		+\sum_{(g,p)\in \mathcal X_{(\mathbf g,\mathbf p)}} s_{\text{ind.}}^{\tilde g} (\tilde p_j,\tilde p_l)(p) 	
				\bar{\mathfrak n}_{(g,p)}(\mathbf g,\mathbf p) - \bar s_{\text{ind.}} M  \epsilon\right\}}
				 \mathbf e_l\\\nonumber
	&\qquad \qquad \qquad -  \mathds 1_{ \left\{\theta \leq s_{\text{nat.}}^{\tilde g} (\tilde p_j,\tilde p_l)
			+\sum_{(g,p)\in \mathcal X_{(\mathbf g,\mathbf p)}} s_{\text{ind.}}^{\tilde g} (\tilde p_j,\tilde p_l)(p) 
				\bar{\mathfrak n}_{(g,p)}(\mathbf g,\mathbf p) + \bar s_{\text{ind.}} M  \epsilon\right\}} \mathbf e_j\:
	 \bigg)\: 
	N_{(\tilde g,\tilde p_j)}^{\text{switch}}(ds,di,d\th, d\tilde p_l),
\end{align}
and similar $X^{2,\e}$ by
\begin{align}
 X^{2,\e}(t)
&\equiv  X^{2,\e}(0)+\sum_{j=1}^{k}\int_0^t \!\int_{\mathbb N_0}\!\int_{\mathbb R_+}
	\mathds 1_{\left\{i\leq X^{2,\e}_{j}({s-}),\; 	
				\theta \leq b(\tilde p_i)+\e\right\}}
	  \mathbf e_j N^{\text{birth}}_{(\tilde g,\tilde p_j)}(ds,di,d\th)\\\nonumber
&
\quad -\sum_{j=1}^{k}\int_0^t \!\int_{\mathbb N_0}\!\int_{\mathbb R_+}
	\mathds 1_{\left\{i\leq X^{2,\e}_{j}({s-}),\; 
					\theta \leq d (\tilde p_j)
						+\sum_{(g,p)\in \mathcal X_{(\mathbf g,\mathbf p)}}c( \tilde p_j, p )  \bar{\mathfrak n}_{(g,p)}(\mathbf g,\mathbf p)-\bar c M\epsilon \right\}}
		 \mathbf e_j  N_{(\tilde g,\tilde p_j)}^{\text{death}}(ds,di,d\th)	\\	\nonumber
&\quad+	\sum_{j=1}^{k}\int_0^t \!\int_{\mathbb N_0}\!\int_{\mathbb R_+}\!\int_{[\tilde p]_{\tilde g}}
	\mathds 1_{\left\{i\leq X^{1,\e}_{j}({s-}) ,\:i\neq j\right\}} \:
	\bigg( \:\mathds 1_{\left\{ \theta \leq s_{\text{nat.}}^{\tilde g} (\tilde p_j,\tilde p_l)
		+\sum_{(g,p)\in \mathcal X_{(\mathbf g,\mathbf p)}} s_{\text{ind.}}^{\tilde g} (\tilde p_j,\tilde p_l)(p) 	
				\bar{\mathfrak n}_{(g,p)}(\mathbf g,\mathbf p) + \bar s_{\text{ind.}} M  \epsilon\right\}}
				 \mathbf e_l\\\nonumber
	&\qquad \qquad \qquad -  \mathds 1_{ \left\{\theta \leq s_{\text{nat.}}^{\tilde g} (\tilde p_j,\tilde p_l)
			+\sum_{(g,p)\in \mathcal X_{(\mathbf g,\mathbf p)}} s_{\text{ind.}}^{\tilde g} (\tilde p_j,\tilde p_l)(p) 
				\bar{\mathfrak n}_{(g,p)}(\mathbf g,\mathbf p) - \bar s_{\text{ind.}} M  \epsilon\right\}} \mathbf e_j\:
	 \bigg)\: 
	N_{(\tilde g,\tilde p_j)}^{\text{switch}}(ds,di,d\th, d\tilde p_l),\end{align}
where  $\mathbf e_j $ is the $j$-th unit vector in $\mathbb R^{k}$ and 
$N^{\text{birth}}$, $N^{\text{death}}$, and $ N^{\text{switch}}$ are the collections of Poisson point measures defined in Subsection \ref{construction}.
Note that
$X^{1,\e}(t)$ and $X^{2,\e}(t)$ are $k$-type
branching processes with the following dynamics: For each  $1\leq i\leq k$, each individual in $X^{1,\e}(t)$, respectively $X^{2,\e}(t)$, with trait $(\tilde g,\tilde p_i)$ undergoes

\begin{enumerate}[(i)]
 \item {birth (without mutation) with rate \:$b(\tilde p_i)-\epsilon$,  respectively  $ b(\tilde p_i)+\e+ 2 (k-1) \bar s_{\text{ind.}}M \epsilon$,\:}
\item {death with rate \:$D_{(\mathbf g,\mathbf p)}(\tilde p_i)  +\bar c M \epsilon+ 2(k-1) \bar s_{\text{ind.}} M \epsilon$,\:  respectively  $D_{(\mathbf g,\mathbf p)}(\tilde p_i) -\bar c M \epsilon$,\\ where $D_{(\mathbf g,\mathbf p)}(\tilde p_i)\equiv  d(\tilde p_i) + \sum_{(g,p)\in\mathcal X_{(\mathbf g,\mathbf p)}} c(\tilde p_i, p) 
\bar{\mathfrak n}_{(g,p)}(\mathbf g,\mathbf p)$, }
\item{ switch to $\tilde p_j$ with rate $S_{(\mathbf g,\mathbf p)}(\tilde p_i,\tilde p_j)- \bar s_{\text{ind.}} M \epsilon$  for all $j\neq i$ (for both processes),\\
where $ 
S_{(\mathbf g,\mathbf p)}(\tilde p_i,\tilde p_j)\equiv s_{\text{nat.}}^{\tilde g}(\tilde p_i, \tilde p_j) +\sum_{( g, p)\in\mathcal X_{(\mathbf g,\mathbf p)}}
				 s^{\tilde g}_{\text{ind.}}(\tilde p_i, \tilde p_j)( p) \bar \frakn_{ ( g, p)}$. }
\end{enumerate} 
Moreover, the processes $X^{1,\e}(t)$ and $X^{2,\e}(t)$ have the following property:
There exists a $K_0>1$ such that  for all ${\tilde p}_i\in [\tilde p]_{\tilde g}$ and for all $K\geq K_0$
\be\label{compare}
\forall \:0\:\leq t\:\leq \theta^{K,\epsilon}_{\text{exit}} \wedge \tau^{}_{\text{mut.}}\wedge \tilde\theta^{K}_{\epsilon}:\qquad
 X^{1,\e}_{i}(t)\leq \nu_t^K(\tilde g,\tilde p_i) K\leq  X^{2,\e}_{i}(t). \qquad 
\ee
Hence, if $ \tilde\theta^{K}_{\epsilon}\leq \theta^{K,\epsilon}_{\text{exit}} \wedge \tau^{}_{\text{mut.}}$, then
\be
\inf\left\{t\geq 0: X^{2,\e}(t)= \lceil \epsilon K\rceil\right\}\leq \tilde\theta^{K}_{\epsilon}\leq \inf\left\{t\geq 0: X^{1,\e}(t)=\lceil \epsilon K\rceil\right\}.
\ee
On the other hand, if $\inf\{t\geq 0: X^{2,\e}(t)=0\}
\leq\inf\{t\geq 0: X^{2,\e}(t)= \lceil \epsilon K\rceil\}\wedge \theta^{K,\epsilon}_{\text{exit}} \wedge \tau^{}_{\text{mut.}}$, then
\be
 \tilde\theta^{K}_{0} \leq \inf\{t\geq 0: X^{2,\e}(t)=0\}.
\ee

Next, let us identify the infinitesimal generator of the control processes $X^{1,\e}$ and $X^{2,\e}$. Therefore, define, for $ i=1,\ldots,k$,
\be\label{apparent-fitness} 
f_{(\mathbf g, \mathbf p)}	{(\tilde g, \tilde p_i)}	\equiv b(\tilde p_i) -D_{(\mathbf g,\mathbf p)}(\tilde p_i) -\textstyle\sum_{j\neq i}S_{(\mathbf g,\mathbf p)}(\tilde p_i,\tilde p_j).
\ee 

($f_{(\mathbf g, \mathbf p)}{(\tilde g, \tilde p_i)}$ would be the invasion fitness of phenotype $\tilde p_i$ if there was no switch back from the
other phenotypes to $\tilde p_i$.)
 Then, by Equation (\ref{inf. gen.}), 
 the infinitesimal generators are given by the following matrixes  

\be \mathbf A ({X^{l,\epsilon}}) =\begin{pmatrix}
f^{l,\e}_{(\mathbf g, \mathbf p)}(\tilde g, \tilde p_1)&	S_{(\mathbf g,\mathbf p)}(\tilde p_1,\tilde p_2)\!-\! \bar s_{\text{ind.}} M \epsilon&	\ldots	&	S_{(\mathbf g,\mathbf p)}(\tilde p_1,\tilde p_k)\!- \!\bar s_{\text{ind.}} M \epsilon\\[0.3em]
S_{(\mathbf g,\mathbf p)}(\tilde p_2,\tilde p_1)\!-\! \bar s_{\text{ind.}} M \epsilon	&	f^{l,\e}_{(\mathbf g, \mathbf p)}(\tilde g, \tilde p_2)&			&		\\[0.3em]
\vdots	&			&	\ddots	&		\vdots		\\[0.3em]
S_{(\mathbf g,\mathbf p)}(\tilde p_k,\tilde p_1)\!-\! \bar s_{\text{ind.}} M \epsilon&		&		\ldots	&f^{l,\e}_{(\mathbf g, \mathbf p)}(\tilde g, \tilde p_k)
\end{pmatrix} \ee
for $l\in\{1,2\}$, where 
$
f^{1,\e}_{(\mathbf g, \mathbf p)}(\tilde g, \tilde p_i)
\equiv  f_{(\mathbf g, \mathbf p)}	{(\tilde g, \tilde p_i)}	- \epsilon(1+ \bar c M + (k-1)\bar s_{\text{ind.}} M )$ and $
f^{2,\e}_{(\mathbf g, \mathbf p)}(\tilde g, \tilde p_i)\equiv   f_{(\mathbf g, \mathbf p)}	{(\tilde g, \tilde p_i)}	+ \epsilon(1+ \bar c M+3(k-1)\bar s_{\text{ind.}} M )$.

We prove in the following 
 that 
the number of mutant individuals  grow with positive probability to $\epsilon K$ before dying out if and only if $\l_{\text{max}}$ of $\mathbf A_{(\tilde g,\tilde p)}\equiv \lim_{\epsilon \to 0} \mathbf A(X^{1,\e})$ is strictly positive. Thus, $\l_{\text{max}}(\mathbf A_{(\tilde g,\tilde p)})$ is an appropriate generalisation of the invasion fitness of the class $[\tilde p]_{\tilde g}$:
\begin{equation}
\label{invasion-fitness-case1}
F_{[\tilde p]_{\tilde g}}(\mathbf g, \mathbf p)\equiv \l_{\text{max}}(\mathbf A_{(\tilde g,\tilde p)}).
\end{equation}
 Since the birth and death rates of $X^{1,\e}$ and $X^{2,\e}$  are positive and since Assumption \ref{recurrence} implies that 
$\mathbf M(X^{1,\e})$ and $\mathbf M(X^{2,\e})$ are irreducible, we obtain that the processes  $X^{1,\e}$ and $X^{2,\e}$ are non-singular and irreducible. Thus, $X^{1,\e}$ and $X^{2,\e}$ satisfy the conditions of Lemma \ref{expansion of Thm 4}.
For $l\in\{1,2\}$, let $\mathbf q (X^{l,\e})$  denote the extinction probability vector  of  $ X^{l,\e}$, i.e.\ 
\be \mathbf q (X^{l,\e})\equiv(q_1 (X^{l,\e}),\ldots, q_k (X^{l,\e})),  \quad\text{where } q_i(X^{l,\e}))\equiv\mathbb P\left[X^{l,\e}(t)=0 \text{ for some }t\; \middle| X^{l,\e}(0)=\mathbf e_i\right].
\nonumber \ee
Observe  that  $\mathbf q (X^{l,\e})=(1,\ldots, 1)$ if $X^{l,\e}$ is not supercritical. To characterise  $\mathbf q (X^{l,\e})$ in the supercritical case,  let us introduce the following functions
 \be \mathbf u^l:[0,1]^k\times(-\eta,\eta) \to \mathbb R^k, \quad \text {where $\eta$ is some small enough constant and $l\in\{1,2\}$,}
 \ee 
 defined, for all $1\leq i\leq k$, by
 \bea\mathbf u^1_i(\mathbf y, \e )\!&\equiv\!&\nonumber
\Bigl(b(\tilde p_i) -\e\Bigr)\:y_i^2+\sum_{j\neq i}\lb S_{(\mathbf g,\mathbf p)}(\tilde p_i,\tilde p_j)- \bar s_{\text{ind.}} M \epsilon \rb\: y_j+D_{(\mathbf g,\mathbf p)}(\tilde p_i)+\bar c M\e + 2(k-1) \bar s_{\text{ind.}} M \epsilon\\&&
-\:\Big( b(\tilde p_i) +\sum_{j\neq i} S_{(\mathbf g,\mathbf p)}(\tilde p_i,\tilde p_j)+D_{(\mathbf g,\mathbf p)}(\tilde p_i)+ \lb1-\bar c M+ (k-1) \bar s_{\text{ind.}} M \rb\epsilon\Big)\:y_i.
\eea  
and 
 \bea\mathbf u^2_i(\mathbf y, \e )\!&\equiv\!&\nonumber
\Bigl( b(\tilde p_i) +\e + 2(k-1) \bar s_{\text{ind.}} M \epsilon\Bigr)\:y_i^2+\sum_{j\neq i}\lb S_{(\mathbf g,\mathbf p)}(\tilde p_i,\tilde p_j)- \bar s_{\text{ind.}} M \epsilon \rb\: y_j+D_{(\mathbf g,\mathbf p)}(\tilde p_i)-\bar c M\e \\&&
-\:\Big( b(\tilde p_i) +\sum_{j\neq i} S_{(\mathbf g,\mathbf p)}(\tilde p_i,\tilde p_j)+D_{(\mathbf g,\mathbf p)}(\tilde p_i)+ \lb1-\bar c M+ (k-1) \bar s_{\text{ind.}} M \rb\epsilon\Big)\:y_i.
\eea  
Observe that $\mathbf u^1 (\mathbf y,\e )$  and $\mathbf u^2 (\mathbf y,\e )$  are  the infinitesimal generating functions of $X^{1,\e}$ and $X^{2,\e}$ and that $\mathbf u^1 (\mathbf y,0 )=\mathbf u^2 (\mathbf y,0 )$. Moreover,
 the extinction vector of a multi-type branching process is  given as the unique root of the
  generating function in the unit cube (cf.\ \cite{A_BP} p.\ 205 or \cite{SEW} Chap.\ 5). Thus,  in the supercritical case
 $\mathbf q (X^{1,\e})$  is   the unique solution of 
 
\be \label{eq_extinct1}
\mathbf u^1(\mathbf y, \e  )=0 \qquad \text{for $\mathbf y \in[0,1)^k$ }
\ee
and $\mathbf q (X^{2,\e})$ is  the unique solution of 
\be \label{eq_extinct2}
\mathbf u^2(\mathbf y, \e  )=0 \qquad \text{for $\mathbf y \in[0,1)^k$. }
\ee
These solutions are in general not analytic. 
Applying Lemma \ref{expansion of Thm 4} to $X^{1,\e}$ and $X^{2,\e}$ we obtain that there exists $C_1>0$ such that, for all $\eta>0$, $\e>0$ sufficiently small and $K$ large enough, 
\bea
\mathbb P\left[\theta_{\text{No Jump}}^{K, M\e}<\tfrac{\eta}{K u_K}\wedge
\theta^{K, M \epsilon}_{\text{exit}} \wedge \tau^{}_{\text{mut.}}\wedge \tilde\theta^{K}_{\epsilon}\right]
&\geq& \mathbb P\left[ \inf\{t\geq 0: X^{2,\e}(t)=0\}<\tfrac{\eta}{K u_K}\right]\qquad\nonumber\\&\geq& q_1(X^{2,\e})-C_1 \e
\eea
and
\bea
\mathbb P\left[\tilde\theta^{K}_{\epsilon}<\tfrac{\eta}{K u_K}\wedge
\theta^{K, M\epsilon}_{\text{exit}} \wedge \tau^{}_{\text{mut.}}\wedge \tilde\theta^{K}_{0}\right]
&\geq& \mathbb P\left[ \inf\{t\geq 0: X^{1,\e}(t)=0\}<\tfrac{\eta}{K u_K}\right]\qquad\nonumber\\&\geq& 1-q_1(X^{1,\e})-C_1 \e.
\eea
If $X^{2,\e}$ is sub- or critical for $\e$ small enough, then $\lim_{\e\downarrow 0}q_1(X^{2,\e})=\lim_{\e\downarrow 0}q_1(X^{1,\e})=1$. In the supercritical case, let $\mathbf q\in[0,1)^k$ be  the solution  of $\mathbf u^1(\mathbf y ,0)=\mathbf u^2(\mathbf y ,0)=0$. 
Then, by applying the implicit function theorem, there exist open sets $U^1\subset \mathbb R$  and $U^2\subset \mathbb R$  containing $0$,
 open sets $V^1\subset \mathbb R^k$ and $V^2\subset \mathbb R^k$ containing $\mathbf q$, and two unique continuously differentiable functions $g^1:U^1\to V^1$  and $g^2:U^2\to V^2$ such that 
\be
\{(\e,g^1(\e))|\e\in U^1 \}=\{(\e,\mathbf y)\in U^1\times V^1| \mathbf u^1(\mathbf y,\e)=0\}.
\ee
and

\be
\{(\e,g^2(\e))|\e\in U^2 \}=\{(\e,\mathbf y)\in U^2\times V^2| \mathbf u^2(\mathbf y,\e)=0\}.
\ee
By definition, $g^1(0)=g^2(0)=\mathbf q$ and $q_1=q_{(\mathbf g, \mathbf p)}(\tilde g,\tilde p)$.
We can linearise and obtain that  there exists a constant $C_2>0$ such that
\be
q_1(X^{1,\e})\leq q_{(\mathbf g, \mathbf p)}(\tilde g,\tilde p) +C_2\e \quad \text{and} \quad  q_1(X^{2,\e})\geq q_{(\mathbf g, \mathbf p)}(\tilde g,\tilde p)-C_2\e
\ee
Therefore, 
\be
\lim_{K\uparrow \infty}\mathbb P\left[\theta_{\text{No Jump}}^{K, M\e}\wedge \tilde\theta^{K}_{\epsilon}<\tfrac{\eta}{K u_K}\wedge
\theta^{K,M\epsilon}_{\text{exit}} \wedge \tau^{}_{\text{mut.}}\right]
\geq 1-
2(C_1 +C_2)\e.
\ee
Conditionally on survival, the proportions of the different phenotypes in  $X^{1,\e}$ converge
almost surely,  as $t \uparrow  \infty$, to the corresponding ratios of the components
of the eigenvector, which are all strictly positive (cf.\ Lem.\  \ref{expansion of Thm 4}, Eq.\ (\ref{limit-ratio})).
 Moreover, there exists a  constant $C_3>0$ such that, for all $\e$ small enough, 
\be
\lim_{K\uparrow \infty}\mathbb P\:\Big[
\Big\{\tilde\theta^{K}_{\epsilon}<\tfrac{\eta}{K u_K}\wedge
\theta^{K, M\epsilon}_{\text{exit}} \wedge \tau^{}_{\text{mut.}}\Big\} \cap\Big \{\inf\{t\geq0: X^{1,\e}(t)=0\}<\infty \Big\}\Big]<C_3\e
\ee
and $\tilde\theta^{K}_{\epsilon}$ converges to infinity as $K\uparrow \infty$.
Thus, conditionally on $\{\tilde\theta^{K}_{\epsilon}<\tfrac{\eta}{K u_K}\wedge
\theta^{K, M\epsilon}_{\text{exit}} \wedge \tau^{}_{\text{mut.}}\}$, there exists a (small) constant $C_4>0$ such that the probability that the densities of the phenotypes $\{\tilde p_1,\ldots,\tilde p_k\}$,  are all larger than $C_4\e$ at time $\tilde\theta^{K}_{\epsilon}$ convergences to one as first $K\uparrow  \infty $ and then $\e\to 0$.
More precisely, there exists constants $C_4>0$ and  $C_5>0$ such that,
 for all $\e$ small enough, 
\be
\lim_{K\uparrow \infty} \label{lower_bound_mutants}
\mathbb P\left[\tilde\theta^{K}_{\epsilon}<\tfrac{\eta}{K u_K}\wedge
\theta^{K, M\epsilon}_{\text{exit}} \wedge \tau^{}_{\text{mut.}},\:\exists i\in\{1,\ldots k\}:\:\nu^K_{\tilde\theta^{K}_{\epsilon}}(\tilde p_i)\leq C_4\e\right]\leq C_5\e.
\ee

\emph{The second invasion step.} By  Assumption \ref{conv_to_fixedpoint}, any solution of $LVS(d+1,((\mathbf g,\mathbf p), (\tilde g,\tilde p)))$ with initial state in the compact set
 \be \label{set A}A\equiv \left\{x\in\mathbb R^{|\mathcal X_{(\mathbf g,\mathbf p)}|}:|x-\bar{\mathfrak n}(\mathbf g,\mathbf p)|\leq M\e\right\}\times [C_4\e,\e]^k\ee converge, as $t \uparrow  \infty$, to the unique locally strictly stable
 equilibrium  $\mathfrak n^{*}((\mathbf g,\mathbf p)),(\tilde g,\tilde p))$. Therefore, for all $\e>0$ there exists  $T(\e)\in\mathbb R$ such that any of these trajectories do not leave the set  
 \be\left\{x\in\mathbb R^{|\mathcal X_{(\mathbf g,\mathbf p)}|+k}:|x-\mathfrak n^{*}((\mathbf g,\mathbf p)),(\tilde g,\tilde p))|\leq \e^2/2\right\}\ee  after time $T(\e)$.
 Back to the stochastic system, let us introduce on the event $\{\tilde\theta^{K}_{\epsilon}<\tfrac{\eta}{K u_K}\wedge
\theta^{K, M\epsilon}_{\text{exit}} \wedge \tau^{}_{\text{mut.}}\}$  the following stopping time
\be
\theta^{K,\e}_{\text{near } \mathfrak{n}^*}=\inf\Big\{t\geq\tilde\theta^{K}_{\epsilon}: \Big|\Big|\nu_t^K-\textstyle\sum_{x\in\mathcal X_{((\mathbf g,\mathbf p),(\tilde g,\tilde p))}} \mathfrak n^*_{x} ((\mathbf g,\mathbf p),(\tilde g,\tilde p)) \delta_{x}\Big|\Big|_{TV} <\e ^2  \Big\}.
\ee 
Then, we conclude by using the strong Markov property at  $\tilde\theta^{K}_{\epsilon}$ and Theorem \ref{exit_domain_prob} (i) on $[0,T(\e)]$ %=\{x\in\mathbb R^{|\mathcal X_{(\mathbf g,\mathbf p)}|}:|x-\bar{\mathfrak n}(\mathbf g,\mathbf p)|\leq M\e\}\times [C_4\e,\e]^k$,
 that there exists a constant $C_6>0$ such that, for all $\e$ small enough,
\bea
\lim_{K\uparrow  \infty}\mathbb P\bigg[\tilde\theta^{K}_{\epsilon}<\tau_{\text{mut.}}\wedge\tfrac{\eta}{K u_K} \:\text{ and }
\sup_{s\in[\tilde\theta^{K}_{\epsilon},\tilde\theta^{K}_{\epsilon}+T(\e)]}\Big|\Big|\nu^K_s-\textstyle\sum_{x\in \mathcal X_{(\mathbf g,\mathbf p)}} \mathfrak n_x(s, \nu^K_0)\delta_{x}\Big|\Big|^{}_{TV}\leq \e^2\bigg]\qquad&&\quad\\\nonumber\geq 1- q_{(\mathbf g, \mathbf p)}(\tilde g,\tilde p)-C_6\e,&&
\eea
which implies
\be
\lim_{K\uparrow  \infty}\mathbb P\left[\tilde\theta^{K}_{\epsilon}<\theta^{K,\e}_{\text{near } \mathfrak{n}^*}<\tau_{\text{mut.}}\wedge\tfrac{\eta}{K u_K}\right]\geq 1- q_{(\mathbf g, \mathbf p)}(\tilde g,\tilde p)-C_6\e.
\ee
We used that, at time ${{\tilde\theta^{K}_{\epsilon}}}$,  the stochastic process $\nu^K$ (considered as element of $\mathbb R^{|\mathcal X_{(\mathbf g,\mathbf p)}|+k}$) lies in the compact set $A$, where $A$ is defined in (\ref{set A}).

\emph{The third invasion step.} After time $\theta^{K,\e}_{\text{near } \mathfrak{n}^*}$ we use again comparisons with multi-type branching processes to show that all individuals carrying a trait which is not present in the new equilibrium  $\mathfrak{n}^* $ die out. To this aim let us define
 \be \mathcal X^{\mathfrak {n}^*}_{\text{extinct}}=\{(g,p)\in \mathcal X_{((\mathbf g,\mathbf p),(\tilde g,\tilde p))}: \mathfrak{n}^*_{( g,  p)}((\mathbf g,\mathbf p),(\tilde g,\tilde p))=0\}
\ee

For proving that the populations with traits in $ \mathcal X^{\mathfrak {n}^*}_{\text{extinct}}$ stay small after $\theta^{K,\e}_{\text{near } \mathfrak{n}^*}$ and  that the populations with traits not in $ \mathcal X^{\mathfrak {n}^*}_{\text{extinct}}$ stay close to its equilibrium value after $\theta^{K,\e}_{\text{near } \mathfrak{n}^*}$, let us define
\be
\theta^{K, \e}_{\text{not small }}=\inf\left\{t\geq \theta^{K,\e}_{\text{near } \mathfrak{n}^*}: \exists (g,p)\in  \mathcal X^{\mathfrak {n}^*}_{\text{extinct}} \text { \:such that\: } \nu^K_t(g,p)> \e\right\}
\ee
and 
\be
\theta^{K, M\e}_{\text{ exit } \mathfrak {n}^*}\equiv
\inf\left\{t\geq \theta^{K,\e}_{\text{near } \mathfrak{n}^*}: \Big|\Big|\nu_t^K-\textstyle\sum_{x\in\mathcal X_{((\mathbf g,\mathbf p),(\tilde g,\tilde p))}} \mathfrak n^*_{x} ((\mathbf g,\mathbf p),(\tilde g,\tilde p)) \delta_{x}\Big|\Big|_{TV}>M \e \right\}.
\ee
By using first the strong Markov property at  $\theta^{K,\e}_{\text{near } \mathfrak{n}^*}$, 
 we can apply Theorem \ref{exit_domain_prob} (ii) and obtain that there exist constants $M>0$ and $C_7>0$ such that, for all $\e$ small enough,
\be
\lim_{K\uparrow  \infty}
\mathbb P\left [ \tilde\theta^{K}_{\epsilon}<\theta^{K,\e}_{\text{near } \mathfrak{n}^*}<\tau_{\text{mut.}}\wedge\tfrac{\eta}{K u_K} \:\text{ and } \:\theta^{K, M\e}_{\text{exit } \mathfrak {n}^*}<\eee^{KV}\!\wedge \tau^{}_{\text{mut.}}\wedge \theta^{K, \e}_{\text{not small }}\right]<C_7 \e 
\ee This is obtained in a similar way as Equation (\ref{exit_domain_apply}) in the first step.
Note that  $(g,p)\in \mathcal X^{\mathfrak {n}^*}_{\text{extinct}}$ implies that  $(g,p_i)\in \mathcal X^{\mathfrak {n}^*}_{\text{extinct}}$ for all $p_i\in[p]_g$,  which is a consequence of Assumption \ref{recurrence}.

Using the same arguments as in the first step, we can construct, 
 for all $(g,p)\in \mathcal X^{\mathfrak {n}^*}_{\text{extinct}}$, a $|[ p]_{ g}|$-type continuous-time
branching process  $Y^{\e, (g,p)}(s)$ with initial condition 
\be Y^{\e, (g,p)}_i(0)=\nu_{ \theta^{K,\e}_{\text{near } \mathfrak{n}*}}^K( g, p_i) K \qquad \text{ for all }p_i\in [ p]_{ g}
\ee
such that, for all $K$ large enough and, for all
 $ t\in[ \theta^{K,\e}_{\text{near } \mathfrak{n}^*}, 
 \theta^{K, M\e}_{\text{exit } \mathfrak {n}^*}\wedge\theta^{K, \e}_{\text{not 
 small }}\wedge \tau_{\text{mut.}}]$,
\be
\nu_{t}^K( g, p_i) K\leq  Y^{\e,(g,p)}_{i}(t-\theta^{K,\e}_{\text{near } \mathfrak{n}*})\qquad \text{ for all }{p}_i\in [ p]_{ g}. 
\ee
Moreover,  $Y^{\e, (g,p)}(t)$ is characterised as follows: 
For each  ${p}_i\in [ p]_{ g}$,  each individual in $Y^{\e, (g,p)}(t)$ with trait $(g, p_i)$ undergoes

\begin{enumerate}[(i)]
\setlength{\itemsep}{6pt}
 \item {birth (without mutation) with rate \:$ b( p_i)+2 ( | [ p]_{ g}|-1)s_{\text{ind.}} (M+| \mathcal X^{\mathfrak {n}^* }_{\text{extinct}}|)\epsilon$\:}
\item {death with rate \:$d( p_i) + \sum_{(\hat g,\hat p)\in\mathcal X_{(\mathbf g,\mathbf p),(\tilde g, \tilde p)}} c( p_i,\hat p) 
\mathfrak n^*_{(\hat g,\hat p)}((\mathbf g,\mathbf p),(\tilde g, \tilde p))-\bar c(M+| \mathcal X^{\mathfrak {n}^* }_{\text{extinct}}|)\:\e$  }
\item{for all $ j\neq i$, switch to $ p_j$ with rate \\
$s_{\text{nat.}}^{g}( p_i, p_j) +\sum_{(\hat g,\hat p)\in\mathcal X_{(\mathbf g,\mathbf p),(\tilde g, \tilde p)}} 				 s^{ g}_{\text{ind.}}( p_i,  p_j)( \hat p)\mathfrak n^*_{(\hat g,\hat p)}((\mathbf g,\mathbf p),(\tilde g, \tilde p))- \bar s_{\text{ind.}} (M+| \mathcal X^{\mathfrak {n}^* }_{\text{extinct}}|)\epsilon$ .}
\end{enumerate} 
Let 
 $\mathbf A(Y^{\e,(g,p)})$ denote  the infinitesimal generator of the process $Y^{\e,(g,p)}$.
 Since the equilibrium ${\mathfrak n^*}((\mathbf g,\mathbf p),(\tilde g, \tilde p))$ is locally strictly stable (cf.\ Ass.\ \ref{conv_to_fixedpoint}), the eigenvalues of the Jacobian matrix of the dynamical system at ${\mathfrak n^*}((\mathbf g,\mathbf p),(\tilde g, \tilde p))$ are all strictly negative. 
If $\e$ is small enough, this implies that all eigenvalues of  $\{\mathbf A(Y^{\e,(g,p)}) , (g,p)\in \mathcal X^{\mathfrak {n}^* }_{\text{extinct}} \}$ are strictly negative.
(There exists an order of the elements of  $\mathcal X_{(\mathbf g,\mathbf p),(\tilde g, \tilde p)} $ such that the Jacobian matrix is an upper-block-triangular matrix and $\{\mathbf A(Y^{0,(g,p)}) , (g,p)\in \mathcal X^{\mathfrak {n}^* }_{\text{extinct}} \}$ are on the diagonal.)  Thus,  for all $\e$  small enough, the branching processes $\{Y^{\e, (g,p)},(g,p)\in \mathcal X^{\mathfrak {n}^*}_{\text{extinct}}\}$ are all subcritical. Moreover, we can apply Lemma \ref{expansion of Thm 4} and get,  for all $\e$  small enough
and $(g,p)\in \mathcal X^{\mathfrak {n}^*}_{\text{extinct}}$
 \be
 \lim_{K\uparrow \infty}\:\mathbb P\left[\inf\{t\geq 0: Y^{\e, (g,p)}(t)=0\}\leq \frac{\eta}{Ku_K}\right]=1,
 \ee
and there exists a constant $C_8$ such that, for all $\e$ small enough and $(g,p)\in \mathcal X^{\mathfrak {n}^*}_{\text{extinct}}$,
\be
\lim_{K\uparrow \infty }
\mathbb P\left[ \inf\{t\geq 0: Y^{\e, (g,p)}(t)=\lceil \e K\rceil \}\leq \inf\{t\geq 
0: Y^{\e, (g,p)}(t)=0\}\right]\leq C_8 \e.
\ee
Hence, there exists a constant $M>0$ and $C_9>0$ such that, for all $\eta>0$ and $\e$ small enough,
\be
\lim_{K\uparrow \infty}\mathbb 
P\left[\tilde\theta^{K}_{\epsilon}<\theta^{K,M\epsilon}_{\text{Jump}}< 
\tau_{\text{mut.}}\wedge\frac{\eta}{Ku_K}\wedge \theta^{K, \e}_{\text{not 
small }}\right ]\geq1- q_{(\mathbf g, \mathbf p)}(\tilde g,\tilde p)-C_9\e,
\ee
which finishes the proof of the theorem.
\end{proof}
Combining all the previous results, we can prove similar as in \cite{C_TSS} that for, all $\e>0, t>0$ and
$\Gamma\subset \mathcal X$,
\begin{align}
\lim_{K\uparrow \infty}\mathbb P\Big[\text{Supp}(\nu^K_{t/Ku_K})=\Gamma, \text{ all traits of $\Gamma$ coexist in $LVS(|\Gamma|, \Gamma)$, } \quad&\\\nonumber
\text{and }||\nu^K_{t/Ku_K}-\sum_{x\in\Gamma}\bar{\mathfrak{n}}_x(\Gamma)\delta_{x}||_{TV}<\e&\Big]
=\mathbb P[\text{Supp}(\Lambda_t )=\Gamma]
\end{align}
where  $\Lambda$ is the PES with phenotypic plasticity defined in Theorem \ref{PESwP}. Finally, generalising this to any sequence of times $0<t_1<\ldots<t_n$, implies that $(\nu^K_{t/Ku_K})_{t\geq 0}$ converges in the sense of finite dimensional
distributions   to $(\Lambda_t)_{t\geq0}$ (cf.\ \cite{C_TSS}, Cor.\ 1 and Lem.\ 1), which ends the proof of Theorem \ref{PESwP}.
%
%%%%%%%%%%%%%%%%%%%%%%%%%%%%%%%%%%%%%%%%%%%%%%%%%%%%%%%%%%%%%%
\subsection{Examples.}
Figure  \ref{mutants} shows  two examples  where in a population consisting only of type $(g,p)$ and being close to $\frakn (g,p)$ a mutation to genotype $\tilde g$ occurs. In these example, $\tilde g$ is associated with two possible phenotypes $\tilde p_1$ and $\tilde p_2$. 

\begin{figure}[h!]
	\centering
	\begin{minipage}{.48\textwidth} 
		A\!\! \!\!\!\includegraphics[width=\textwidth, height=140px]{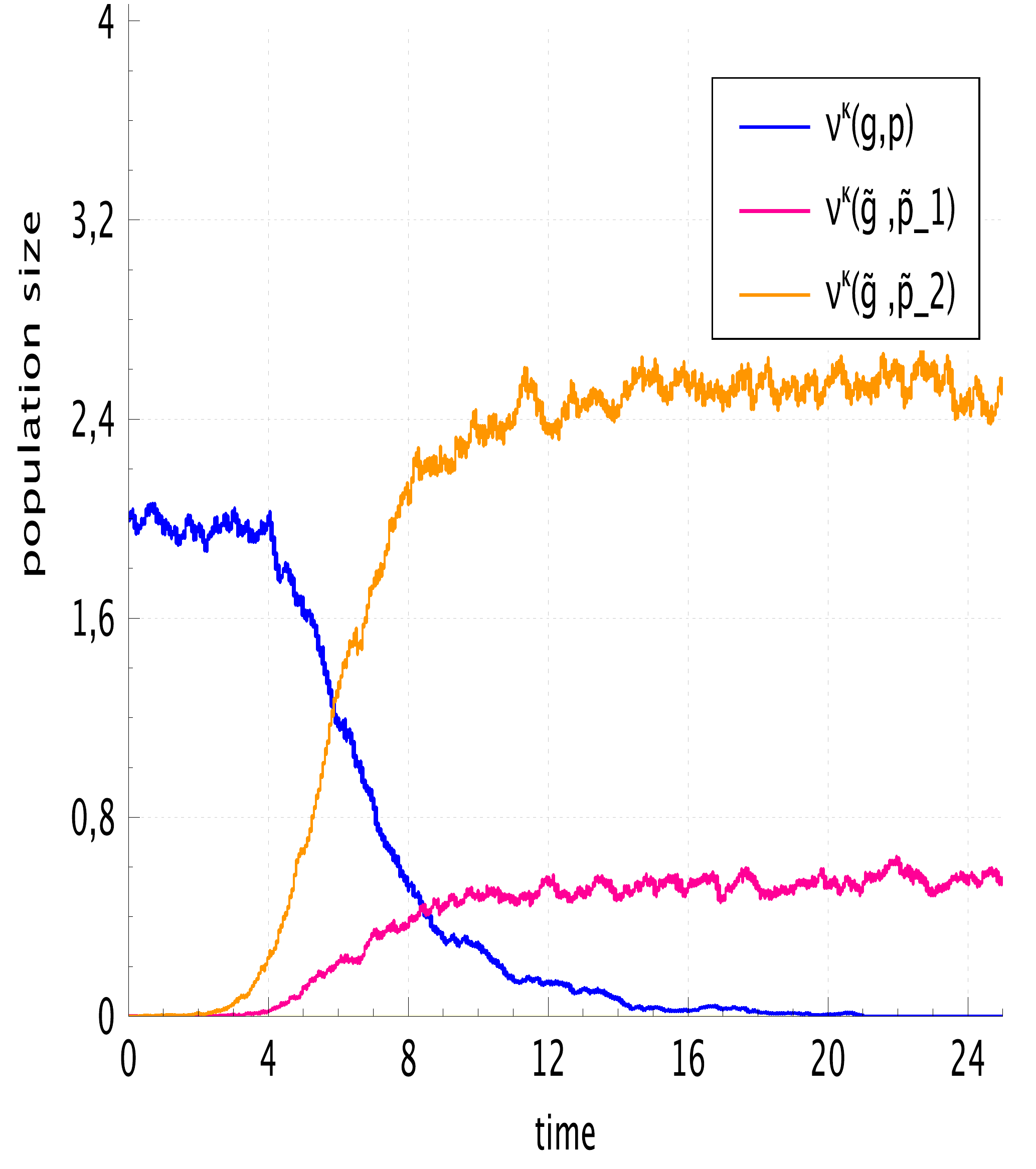}
	\end{minipage}\quad
	\begin{minipage}{.48\textwidth} 
		B \!\!\!\!\!\includegraphics[width=\textwidth, height=140px]{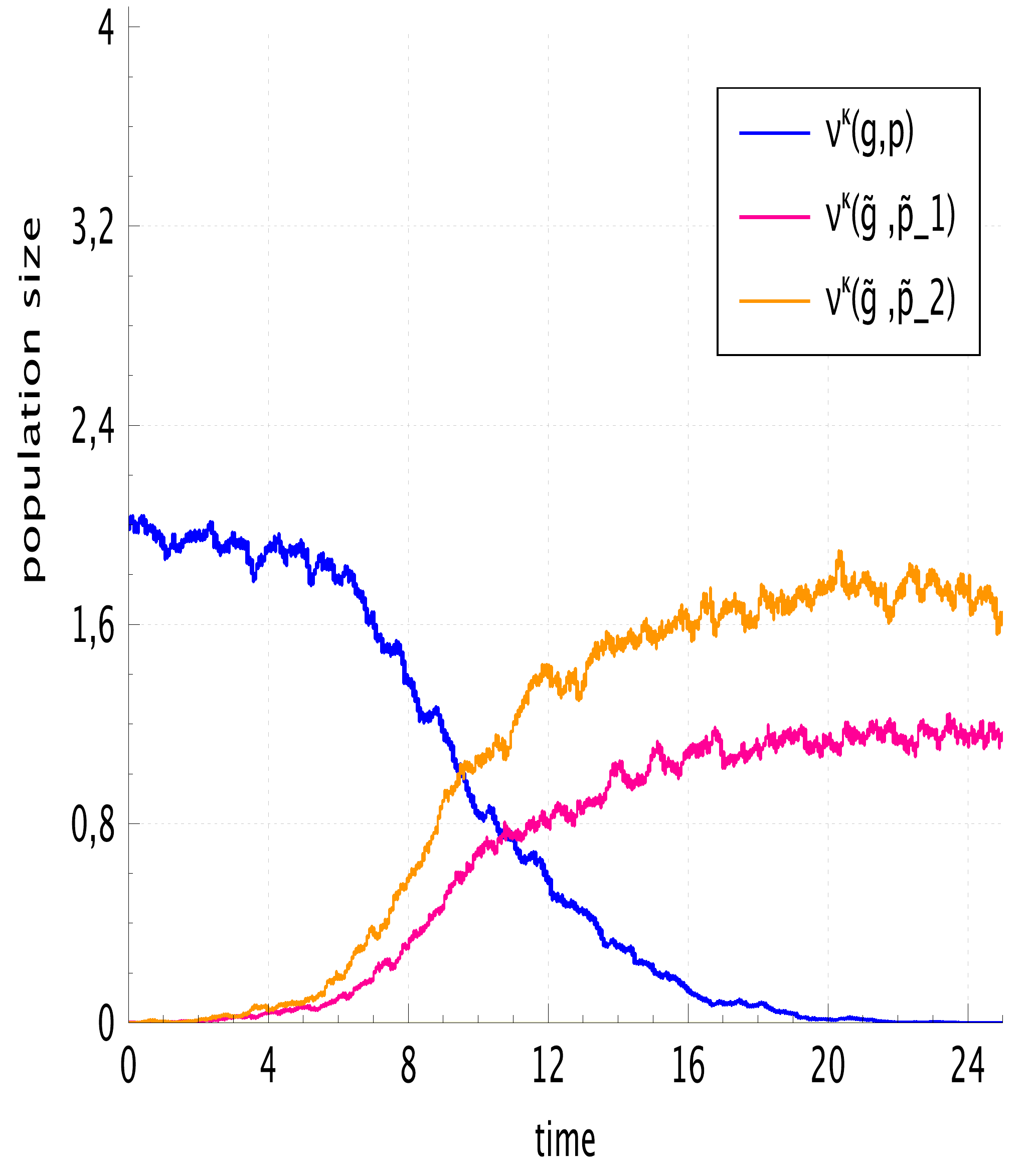}
	\end{minipage}
	
	\caption{\small {Simulations of the invasion phase with $K=1000$.
			(A) The mutant phenotype $\tilde p_1$ has a negative initial growth rate but can switch to $\tilde p_2$  which has a positive one. The fitness of the genotype $\tilde g$ is positive.
			(B) The fitness of the mutant genotype $\tilde g$ is positive, although 
			each phenotype has a 
			negative initial growth rate. This is possible because an outgoing switch is a loss of a cell for a phenotype, but not for the 
			whole genotype. 
			\label{mutants}} }
\end{figure}

In example (A), we start with a single mutant carrying trait $(\tilde g,\tilde p_1)$ and which can switch to $\tilde p_2$ but the back-switch is relative weak (cf.\ Tab.\ \ref{TableExaA}). According to definition \eqref{apparent-fitness} we have $f_{(g,p)}(\tilde g,\tilde p_1)<0$ and $f_{(g,p)}(\tilde g,\tilde p_2)>0$. 
However, the global fitness of the genotype $\tilde g$ is positive. More 
precisely, it is given by the largest eigenvalue of $\lb\begin{smallmatrix} -3 
&2\\ 0.6 &1 \end{smallmatrix}\rb$, which equals approximatively $1.280$. 
Therefore, the multi-type branching process approximating the mutant 
population in the first step is supercritical. This does not depend on  the 
phenotype of the first mutant, i.e.\  we would have the same if we had 
started with a single mutant carrying trait $(\tilde g,\tilde p_2)$). However, 
the probability of invasion depends this. In this example, the invasion 
probability is given by the solution of
\begin{align}
2y_1^2+2y_2+3-7y_1&\;=\;0,\\
4y_2^2+0.6y_1+2.4-7y_2&\;=\;0.
\end{align}
Thus, if we start with the trait $(\tilde g,\tilde p_1)$, the invasion probability is approximately 
$0.199$. Whereas it is $0.338$ if the first one has trait $(\tilde g,\tilde p_2)$. In Figure \ref{mutants} (A), the mutant population with genotype $\tilde g$ survives and the stochastic process is attracted to the new equilibrium $\frakn^*((g,p),(\tilde g, \tilde p_1), (\tilde g, \tilde p_2))\approx(0,0.543,2.554)$, which is a strictly stable.  

\begin{table}[h]
\centering\small
\begin{tabular}{|@{\hspace{ 0.1cm}}l@{\hspace{ 0.1cm}} |@{\hspace{ 0.1cm}}l @{\hspace{ 0.1cm}} |@{\hspace{ 0.1cm}}l@{\hspace{ 0.1cm}}|@{\hspace{ 0.1cm}}l@{\hspace{ 0.1cm}}| @{\hspace{ 0.1cm}}l @{\hspace{ 0.1cm}}|@{\hspace{ 0.1cm}}l @{\hspace{ 0.1cm}}||@{\hspace{ 0.1cm}}l@{\hspace{ 0.1cm}}|}\hline
$b(p)=3	  $ &$ d({p})=1	$ &$c({p},{p})=1	  $&$ c({p},\tilde p_1)=1$&$ c(p,\tilde p_2)=0.7$&$s_{\text{ind.}}^{.} (\, .\,, .\,)(.)\equiv 0 $&$\nu^K_0{(g,p)}=2$ \\[0.1em]\hline
$b(\tilde p_1)=2$&$	 d(\tilde p_1)=1	$&$	c(\tilde p_1,{p})=1 $&$ c(\tilde p_1,\tilde p_1)=1$&$c(\tilde p_1,\tilde p_2)=0.5$&$s^{\tilde g}(\tilde p_1,\tilde p_2)=2$&$\nu^K_0{(\tilde g,\tilde p_1)}=K^{-1}$\\[0.1em]\hline
$b(\tilde p_2)=4$&$	 d(\tilde p_2)=1	$&$	c(\tilde p_2,{p})=0.7 $&$ c(\tilde p_2,\tilde p_1)=0.5$&$c(\tilde p_2,\tilde p_2)=1$&$s^{\tilde g}(\tilde p_2,\tilde p_1)=0.6$&$\nu^K_0{(\tilde g,\tilde p_2)}=0	$\\[0.1em]\hline
\end{tabular}
\caption{\small Parameters of Figure \ref{mutants} (A) }\label{TableExaA}
\end{table}

In example (B), $f_{(g,p)}(\tilde g,\tilde p_1)$ and $f_{(g,p)}(\tilde g,\tilde p_2)$  are both negative. Nevertheless, the fitness of the genotype is positive and thus the mutant
invades with positive probability. (It is given by the largest eigenvalue of $\lb\begin{smallmatrix} -3 &2\\ 2 &-0.4 \end{smallmatrix}\rb$, which equals approximatively $0.685$.) However, the invasion probability is smaller in this example.  It is approximately 
$0.127$ if we start with the trait $(\tilde g,\tilde p_1)$ and  $0.207$ else. In Figure \ref{mutants} (B), the mutant population survives and the process is attracted to the stable fixed point $\frakn^*((g,p),(\tilde g, \tilde p_1), (\tilde g, \tilde p_2))\approx(0,1.153,1.745)$.  
Hence, this examples illustrate that the usual definition of invasion fitness fails for populations with phenotypic plasticity.
\begin{table}[h]
\centering\small
\begin{tabular}{|@{\hspace{ 0.1cm}}l@{\hspace{ 0.1cm}} |@{\hspace{ 0.1cm}}l @{\hspace{ 0.1cm}} |@{\hspace{ 0.1cm}}l@{\hspace{ 0.1cm}}|@{\hspace{ 0.1cm}}l@{\hspace{ 0.1cm}}| @{\hspace{ 0.1cm}}l @{\hspace{ 0.1cm}}|@{\hspace{ 0.1cm}}l @{\hspace{ 0.1cm}}||@{\hspace{ 0.1cm}}l@{\hspace{ 0.1cm}}|}\hline
$b(p)=3	  $ &$ d({p})=1	$ &$c({p},{p})=1	  $&$ c({p},\tilde p_1)=1$&$ c(p,\tilde p_2)=0.7$&$s_{\text{ind.}}^{.} (\, .\,, .\,)(.)\equiv 0 $&$\nu^K_0{(g,p)}=2$ \\[0.1em]\hline
$b(\tilde p_1)=2$&$	 d(\tilde p_1)=1	$&$	c(\tilde p_1,{p})=1 $&$ c(\tilde p_1,\tilde p_1)=1$&$c(\tilde p_1,\tilde p_2)=0.5$&$s^{\tilde g}(\tilde p_1,\tilde p_2)=2$&$\nu^K_0{(\tilde g,\tilde p_1)}=1/K$\\[0.1em]\hline
$b(\tilde p_2)=4$&$	 d(\tilde p_2)=1	$&$	c(\tilde p_2,{p})=0.7 $&$ c(\tilde p_2,\tilde p_1)=0.5$&$c(\tilde p_2,\tilde p_2)=1$&$s^{\tilde g}(\tilde p_2,\tilde p_1)=2$&$\nu^K_0{(\tilde g,\tilde p_2)}=0	$\\[0.1em]\hline
\end{tabular}
\caption{\small Parameters of Figure \ref{mutants} (B) }
\end{table}
	
\bibliographystyle{abbrv}
\bibliography{bibliography}

\end{document}